\RequirePackage{fix-cm}
\documentclass[smallextended]{svjour3}       
\smartqed  

\usepackage{amssymb,graphicx,amscd,setspace,verbatim,amsmath,color,algpseudocode,algorithm,mathabx,float}

\newcommand{\braces}[1]{\left\{ #1 \right \}}

\newcommand{\norm}[1]{\left \| #1 \right \|}
\newcommand{\conv}{\operatorname{conv}}

\DeclareMathOperator*{\argmin}{argmin}
\DeclareMathOperator*{\argmax}{argmax}

\newcommand{\TheTitle}{A parallelizable augmented Lagrangian method applied to large-scale non-convex-constrained optimization problems}

\begin{document}

\title{{\TheTitle}\thanks{This work was supported by the Australian Research Council (ARC) grant ARC DP140100985.}}

\titlerunning{A parallelizable augmented Lagrangian method...}        


\author{  
 Natashia Boland
    \and
  Jeffrey Christiansen
  \and
  Brian Dandurand
  \and
  Andrew Eberhard
  \and
  Fabricio Oliveira
}


\institute{
N. Boland \at Georgia Institute of Technology, Atlanta, , USA 
\and
J. Christiansen \at RMIT University, Melbourne, Victoria, Australia 
\and
B. Dandurand \at RMIT University, Melbourne, Victoria, Australia 
\and
A. Eberhard \at RMIT University, Melbourne, Victoria, Australia \\
              Tel.: +61-3-9925-2616\\
              Fax: +61-3-9925-1748\\
              \email{andy.eberhard@rmit.edu.au}           
\and
F. Oliveira \at RMIT University, Melbourne, Victoria, Australia 
}


\date{Received: date / Accepted: date}

\maketitle

\begin{abstract}
We contribute improvements to a Lagrangian dual solution approach applied to large-scale optimization problems whose objective functions are convex,  continuously differentiable and possibly nonlinear, while the non-relaxed constraint set is compact but not necessarily convex. Such problems arise, for example, in the split-variable deterministic reformulation of stochastic mixed-integer optimization problems. The dual solution approach needs to address the nonconvexity of the non-relaxed constraint set while being efficiently implementable in parallel.
We adapt the augmented Lagrangian method framework to address the presence of nonconvexity in the non-relaxed constraint set and the need for efficient parallelization. The development of our approach is most naturally compared with the development of proximal bundle methods and especially with their use of serious step conditions. However, deviations from these developments allow for an improvement in efficiency with which parallelization can be utilized. 
Pivotal in our modification to the augmented Lagrangian method is the use of an integration of approaches based on the simplicial decomposition method (SDM) and the nonlinear block Gauss-Seidel (GS) method. 
An adaptation of a serious step condition associated with proximal bundle methods allows for the approximation tolerance to be automatically adjusted. 
Under mild conditions  optimal dual convergence is proven, and we report computational results on test instances from the stochastic optimization literature.
We demonstrate improvement in parallel speedup over a baseline parallel approach.

\end{abstract}

%

\section{Introduction and Background}
We develop a dual solution approach 
to
the problem of interest having the form
\begin{equation}\label{P1}
\zeta^*:=\min_{x,z} \braces{f(x) : Qx=z, x \in X, z \in Z},
\end{equation}
where $f$ is convex and continuously differentiable, $Q \in \mathbb{R}^{q \times n}$ is a block-diagonal matrix determining linear constraints $Qx=z$, $X\subset \mathbb{R}^n$ is a closed and bounded set, and $Z\subset \mathbb{R}^q$ is a linear subspace. The vector $x \in X$ of decision variables  is derived from the original decisions associated with a problem, while the vector $z \in Z$ of auxiliary variables are introduced to effect a decomposable structure in~\eqref{P1}. The block diagonal components of $Q$ are denoted $Q_i \in \mathbb{R}^{q_i \times n_i}$, $i=1,\dots,m$.
Problem~\eqref{P1} is general enough to subsume, for example, the split-variable deterministic reformulation of a stochastic optimization problem with potentially multiple stages, as defined, for example, in~\cite{BirgeLouveaux2011}, while it can also model the case where $f$ is nonlinear (and convex) and/or $X$ is any compact (but not necessarily convex) set.

The Lagrangian dual function resulting from the relaxation of $Qx=z$ is 
\begin{equation}\label{PhiDefn}
\phi(\omega):= \min_{x,z} \braces{f(x) +\omega^\top( Qx-z), x \in X, z \in Z}. 
\end{equation}
Given that $X$ is compact and $f$ is continuous, in order for $-\infty < \phi(\omega)$ to hold,
it is necessary and sufficient that the following dual feasibility assumption be maintained:
\begin{equation}\label{EqDualFeasAssumpt}
\omega \in Z^\perp := \braces{\upsilon \in \mathbb{R}^{q} : \upsilon^\top z = 0 \;\text{for all}\; z \in Z},
\end{equation}
either by assumption or by construction. 
Under condition~\eqref{EqDualFeasAssumpt}, the $z$ term in definition~\eqref{PhiDefn} vanishes, and we may compute $\phi(\omega)=\min_{x} \braces{f(x) +\omega^\top Qx : x \in X}.$  Consequently, $\phi$ becomes separable as
$$
\phi(\omega) = \sum_{i=1}^m \phi_i (\omega_i)
$$
where $\phi_i(\omega_i) := \min_x \braces{f_i(x_i) + \omega_i^\top Q_i x_i : x_i \in X_i}$ and $\omega = (\omega_1,\dots,\omega_m) \in \mathbb{R}^{q_1} \times \cdots \times \mathbb{R}^{q_m}$ has a block structure compatible with the block diagonal structure of $Q$.
The Lagrangian dual problem is:
\begin{equation}\label{P2}
\zeta^{LD}:=\max_\omega \phi(\omega).
\end{equation}
In this paper, we develop, analyze, and apply an iterative solution approach to solving problem~\eqref{P2} subject to the following challenges:
\begin{description}
\item [{\bf Implementability:}] The set $X$ is not convex (for example, it may have mixed-integer constraints as part of its definition). 
Consequently the augmented Lagrangian method is not supported by the theory of proximal point methods. 
\item [{\bf Efficiency of parallelization:}]
The  solution approach should be amenable to efficient parallel computation, in the sense of maximizing the computational work that can be parallelized, the memory usage that can be distributed, and minimizing the amount of parallel communication. 
\end{description}

For the Lagrangian dual problem~\eqref{P2}, we note that the objective function $\phi$ is concave, even when $f$ and $X$ are not convex. 
We can apply a subgradient method (see e.g. \cite{Shor1985}; in textbooks~\cite{Bertsekas1999,Ruszczynski2006}) for solving~\eqref{P2} in an efficiently parallelizable manner. Such an approach is proposed in~\cite{caroe1999dual}.
However, it is preferable to make use of structural features of~\eqref{P2} that allow for smoothing or regularization, so that better convergence properties are realized. For this reason, we consider alternative developments based on proximal point methods that are modified to address both of the above two challenges.
 

As a starting point, we first consider the classical augmented Lagrangian method based on proximal point methods. 
The augmented Lagrangian (AL) method (also known as the method of multipliers) is developed from proximal point methods, and references include~\cite{Hestenes1969,Powell1969,Bertsekas1982,Bertsekas1999}. 
The AL method typically has favorable convergence properties as a dual solution approach for convex problems (linear convergence rate under certain assumptions, see~\cite{Rockafellar1976MO,Bertsekas1982} and references cited therein). However, two issues arise: 1) 
the set $X$ is not convex, and so current theories of convergence are not applicable;
and 2) the primal subproblem associated with each iteration of the AL method is not separable due to the augmented Lagrange term, making efficient parallel implementations difficult to develop.


We introduce modifications to the AL method that address both of these issues. In order to introduce computational tractability in light of the possible nonlinearity of $f$ and the nonconvexity of $X$, the modified AL method solves an alternative dual problem that can provide a weaker dual bound than that provided by the value of~\eqref{P2}.
In the case when $f$ is linear, the alternative dual problem is equivalent to~\eqref{P2}. This matter is explained in  more detail in Section~\ref{Sect3}.
The method that results from these modifications is most naturally compared with the proximal bundle method.


The proximal bundle method initially appeared in~\cite{LemarechalWolfe1975}, and for a survey with history, see~\cite{OliveiraSagastizabal2014}. 
Use of inexact oracles for computing $\phi(\omega)$ and elements of the subdifferential set $\partial \phi(\omega)$ are studied 
in~\cite{OliveiraSagastizabal2014,OliveiraEtAl2014,HareEtAl2016} and references therein.
In its dual form, the bundle method may be referred to as the stabilized column generation method~\cite{BenAmor2009SCG} or the proximal simplicial decomposition method~\cite{Bertsekas2015}. In implementation, the developed algorithm more closely resembles the latter dual form.

For parallelization of the proximal bundle method, see~\cite{FischerHelmberg2014} and~\cite{LubinEtAl2013}. The approach developed in this paper is most naturally compared with~\cite{LubinEtAl2013}, as both approaches address the manner in which the same continuous master problem is approximately solved. The approach of~\cite{FischerHelmberg2014} uses a substantially different parallel computational paradigm based on subspace optimization. This approach, in which solution subspaces are assigned to processors based on periodically updated global state information, is not necessarily based on the problem's decomposable structure.

The proximal bundle method approach  requires modification for efficient parallelization. This matter is addressed in~\cite{LubinEtAl2013}, where a solution to the continuous master problem is obtained by primal dual interior point methods that exploit the decomposable structure present in the augmented Lagrangian term. We provide and analyze an alternative approach based on the use of:
\begin{enumerate}
\item the simplicial decomposition method (SDM)~\cite{Holloway1974,Hohenbalken1977,Bertsekas1999,Bertsekas2015Text}, which provides an alternative framework to the proximal bundle method to address the implementability of the proximal point method while allowing for the possibility that $f$ is nonlinear; and 
\item nonlinear block Gauss-Seidel (GS) 
method~\cite{Hildreth1957,Warga1963,GS2000,Tseng2001,Bonettini2011} 
to approximate the solutions to the continuous master problem.
\end{enumerate}
Motivated by its constituent parts, the algorithm we develop is referred to as SDM-GS-ALM.

Algorithm SDM-GS-ALM addresses the solution to an alternative dual problem which is equivalent to~\eqref{P2} when $f$ is linear, but in general provides a weaker dual bound otherwise. This dual problem is used to address the more general setting where $f$ is convex but possibly nonlinear.

In an iteration of SDM-GS-ALM, the analog to the continuous master problem is not solved to (near) exactness; instead, approximate solutions based on possibly just one nonlinear block GS iteration are used. Due to the underlying need for convexification of the non-relaxed constraint set, implementability requires that the nonlinear block GS method must be integrated with the SDM so that optimal convergence of the resulting iterations can be established. In this way, a serious step condition similar to that used in proximal bundle methods is eventually satisfied after a finite number of such integrated SDM-GS iterations, and analogous dual optimal convergence of our approach is recovered even with the deviations from the proximal bundle method. In summary, we algorithmically integrate the AL method, the SDM, nonlinear block GS iterations, and the proximal bundle method serious step condition. A convergence analysis is also provided for SDM-GS-ALM. Such an integration allows for a considerable improvement in parallel efficiency with respect to maximizing the computational work that can be parallelized, the memory usage that can be distributed, and minimizing the amount of parallel communication.


Other methods developed in the past that are related to aspects of our contribution  include the following. In terms of approximating within the AL method, we include reference to~\cite{Eckstein2003,EcksteinSilva2013}, where the research goal of developing implementable approximation criteria is addressed. 
The separable augmented Lagrangian (SALA) method~\cite{Hamdi1997}, which is an application of the alternating direction method of multipliers (ADMM)~\cite{GlowinskiMarrocco1975,GabayMercier1976,BoydEtAl2011} with a form of resource allocation decomposition and incorporates separability into the AL method. Other approaches to introducing separability into the AL method include~\cite{ChatzipanagiotisEtAl2014,TappendenEtAl2015}. Jacobi iterate approaches applied within either a proximal bundle method or an AL method framework are considered in~\cite{MulveyRuszczynski1992,Ruszczynski1995}; the accelerated distributed augmented Lagrangian method (ADAL) developed in~\cite{ChatzipanagiotisEtAl2014} is like a Jacobi-iterate analogue of ADMM with supporting convergence analysis. 
Other approaches to incorporating separability are found in the alternating linearization approaches~\cite{KiwielEtAl1999,LinEtAl2014} and the predictor corrector proximal multiplier (PCPM) 
methods~\cite{ChenTeboulle1994,HeEtAl2002}. 
All of these methods provide implementable mechanisms for approximating primal subproblem solutions and effecting parallelism in a setting where $X$ is convex. However, they are not practically implementable in our setting where $X$ is not convex and its convex hull $\conv(X)$ is not given beforehand in a computationally useful closed-form description. 

Another recently developed algorithm, referred to as FW-PH~\cite{BolandEtAl2016}, is closely related to the SDM-GS-ALM algorithm developed in this paper. 
In terms of functionality, both appear as modifications to ADMM with inner approximated subproblem solutions.
While the algorithms differ only slightly in terms of functionality, there are substantial differences in the motivation and the convergence analysis. The convergence analysis of FW-PH interfaces with the convergence analysis for ADMM,
which is most naturally developed in the context of the theory of maximal monotone operators and Douglas-Rachford splitting methods~\cite{EcksteinBertsekas1992,EcksteinYao2015}, or as the proximal decomposition of the graph of a maximal monotone operator~\cite{MaheyEtAl1995}. In contrast, the convergence analysis of SDM-GS-ALM naturally reflects its synthesis of SDM, the nonlinear block GS method, the proximal bundle method, and the AL method. The convergence analysis of SDM-GS-ALM follows under more general assumptions than that for FW-PH. In particular, the convergence analysis of SDM-GS-ALM allows for trimming of the inner approximations, and it does not require the warm-starting required by FW-PH. The most important difference in functionality is due to the influence of ideas from proximal bundle methods in SDM-GS-ALM, where updates of $\omega$ are taken conditionally at each iteration, while such updates are taken unconditionally at each iteration of FW-PH. We shall see that these conditional updates help to mitigate performance problems that arise due to the seemingly inevitable use of suboptimal algorithm parameters.

In papers such as~\cite{GadeEtAl2016,FeizollahiEtAl2015}, ADMM is applied directly to the primal problem~\eqref{P1}. In both works, it is acknowledged that ADMM is not theoretically supported in optimal convergence due to the lack of convexity of $X$. Nevertheless,~\cite{GadeEtAl2016} reports the potential for Lagrangian dual bounds to be recovered at each iteration of ADMM even though it is applied to~\eqref{P1}. In~\cite{FeizollahiEtAl2015}, where ADMM is applied to nonconvex decentralized unit commitment problems, heuristic improvements to ADMM are introduced to address the lack of convexity due to the mixed-integer constraints. In contrast to both of these approaches, where ADMM is applied directly to the primal problem~\eqref{P1}, the approach developed in this paper, and its related approach~\cite{BolandEtAl2016}, both resemble ADMM but with application to a primal characterization of the dual problem.
In these two approaches, the challenge of not having an explicit form for this primal characterization is addressed. 

The remainder of the paper is organized as follows. In Section~\ref{Sect2}, a general algorithmic framework based on the AL method with approximate subproblem solutions is developed and analyzed. In Section~\ref{Sect3}, a specific implementation of the Section~\ref{Sect2} framework is posed based on the integration of SDM and GS methods, which addresses the aforementioned issues of implementability and efficiency of parallelization. In Section~\ref{Sect4}, computational experiments and their outcomes are described and interpreted. And at last, Section~\ref{SectConclusion} concludes the paper and provides avenues for future work. 

\section{An alternative AL approximation approach}\label{Sect2}
In the following development, we address the solution of a slightly different dual problem 
\begin{equation}\label{P2C}
\zeta^{CLD}:=\max_{\omega \in Z^\perp} \phi^C(\omega).
\end{equation}
based on the dual function $\phi^C\left( \omega \right) := \min_x \braces{ f(x) + \omega^\top Qx : x \in \conv(X)}$.
The only difference between $\phi^C$ and $\phi$ is the use of constraint set $\conv(X)$ in the former versus the use of $X$ in the latter. 
We assume as before that $\omega \in Z^\perp$.
\begin{remark}
Under the assumption that $\conv(X)$ is not known beforehand by any characterization, direct evaluation of $\phi^C$ or any of its subgradients at any $\omega \in Z^\perp$ is not possible. 
This dual function is not used in the proximal bundle method and is only treated indirectly in the current development.
\end{remark}

The dual problem~\eqref{P2C} has the following primal characterization
\begin{equation}\label{P1ConvB}
\zeta^{CLD}=\min_{x,z} \braces{f(x) : Qx=z, x \in \conv(X), z \in Z},
\end{equation}
where $\conv(X)$ is the convex hull of $X$. In addition to generating a sequence $\braces{\omega^k}$ of dual solutions to~\eqref{P2C}, our algorithm will also generate a sequence of primal solutions $\braces{(x^k,z^k)}$ to~\eqref{P1ConvB}, and so reference to~\eqref{P1ConvB} will be useful. In applying the AL method to problem~\eqref{P1ConvB}, 
the continuous master problem for fixed $\omega \in Z^\perp$ takes the form
\begin{equation}\label{EqQMPConv}
\zeta_\rho^{AL}(\omega):=\min_{x,z} \braces{L_\rho(x,z,\omega), x \in \conv(X), z \in Z}
\end{equation}
where the augmented Lagrangian (AL) relaxes $Qx=z$ and is defined by
\begin{equation}\label{ALDefn}
L_\rho(x,z,\omega):= f(x) + \omega^\top Qx + \frac{\rho}{2}\norm{Qx-z}_2^2.
\end{equation}
\begin{lemma}\label{LemmaALOpt}
For any optimal solution $\omega^*$ to problem~\eqref{P2C}, we have $\zeta_\rho^{AL}(\omega^*) = \zeta^{CLD}$.
Additionally, any optimal solution $(x^*,z^*)$ to problem~\eqref{EqQMPConv} with $\omega=\omega^*$ is also optimal for problem~\eqref{P1ConvB}.
\end{lemma}
\begin{proof}
We specialize developments in, e.g., Section 4 of~\cite{Rockafellar1976ALMethod} or Section 6.4.3 of~\cite{Ruszczynski2006}. 
Due to the convexity of $f$, $\conv(X)$, and $Z$, we may compute
\begin{align}
&\max_{\omega \in Z^\perp} \phi^C(\omega) - \frac{1}{2\rho}\norm{\omega-\overline{\omega}}_2^2 \label{ProxPtProb}\\
&=\max_{\omega \in Z^\perp} \min_{x}  \braces{f(x)+\omega^\top Qx  -\frac{1}{2\rho} \norm{\omega - \overline{\omega}}_2^2 : x \in \conv(X)} \nonumber \\
&= \min_{x}  \braces{\begin{array}{l} f(x)+\overline{\omega}^\top Qx \\+\max_{\omega \in Z^\perp}\braces{ (\omega - \overline{\omega})^\top Qx-\frac{1}{2\rho} \norm{\omega - \overline{\omega}}_2^2} : x \in \conv(X) \end{array}} \nonumber \\ 
&= \min_{x}  \braces{f(x)+\overline{\omega}^\top Qx   +\frac{\rho}{2}\min_z\braces{\norm{Qx-z}_2^2   :  z \in Z} : x \in \conv(X)} \nonumber \\
&= \min_{x,z} \braces{L_\rho(x,z,\overline{\omega}), x \in \conv(X), z \in Z}. \label{RHSProb}
\end{align}
The switching of $\min$ and $\max$ is justified by the Sion min-max theorem.
In substituting $\overline{\omega}=\omega^*$, the value of the left-hand side maximization problem~\eqref{ProxPtProb} is clearly $\zeta^{CLD}$, while the same substitution on the right-hand side~\eqref{RHSProb} yields the value $\zeta_\rho^{AL}(\omega^*)$, from which we see that $\zeta^{CLD}=\zeta_\rho^{AL}(\omega^*)$. To prove the last claim,
we note that $L_\rho(x^*,z^*,\omega^*) = \zeta^{CLD}$ implies that $\norm{Qx^*-z^*}_2^2=0$. Otherwise, $\phi^C(\omega^*) < \zeta^{CLD}$, contradicting the dual optimality of $\omega^*$. Thus, $(x^*,z^*)$ is feasible and optimal for problem~\eqref{P1ConvB}.
\end{proof}


It is straightforward from the definitions that $\phi^C(\omega) \le \phi(\omega)$ for all dual feasible $\omega \in Z^\perp$. 
In the case when $f$ is linear, we have $\phi^C(\omega) = \phi(\omega)$ for all $\omega \in Z^\perp$ and so $\zeta^{LD}=\zeta^{CLD}$.
But in the general case where $f$ is nonlinear, the dual~\eqref{P2C} can be ``weaker'' than~\eqref{P2}, where $\zeta^{CLD} < \zeta^{LD}$ can occur,
which we see in the following example.
Let $f : \mathbb{R}^2 \mapsto \mathbb{R}$ be defined by $f(x)=(x_1-0.5)^2+(x_2-0.5)^2$, $X=\braces{0,1} \times \braces{0,1}$, and let $Qx=z$ be defined to model the constraints $x_1 -z_1 = 0$ and $x_2 - z_2 = 0$ where $Z = \braces{(z_1,z_2) : z_1=z_2} \subset \mathbb{R}^2$. We see trivially that $\zeta^{CLD} = 0$, which is verified with the saddle point $x_1^*=x_2^*=z_1^*=z_2^*=0.5$ and $\omega^*=(0,0)$. However,
$\zeta^{LD} = 0.5$, which is verified with either of the saddle points $x_1^*=x_2^*=z_1^*=z_2^*=0$ and $\omega^*=(0,0)$, or $x_1^*=x_2^*=z_1^*=z_2^*=1$ and $\omega^*=(0,0)$. Thus, $\zeta^{CLD} < \zeta^{LD}$.

In the proximal bundle method, the dual function $\phi$ is approximated by a cutting plane model function which majorizes $\phi$. 
In the next development, we use the following approximation $\widehat{\phi} : \mathbb{R}^q \times \mathbb{R}^n \times \mathbb{R}^q \mapsto \mathbb{R}$ of $\phi^C$ centered at $(x^k,z^k)$, $k \ge 0$, in place the cutting plane model:
\begin{align*}
\widehat{\phi}(\omega,x^k,z^k) := L_\rho(x^k,z^k,\omega) + \frac{\rho}{2} \norm{Qx^k-z^k}_2^2.
\end{align*}
This approximation satisfies the following bounding relationship.
\begin{lemma}\label{LemmaMajBound}
For each $(x^k,z^k)$, $k \ge 0$, such that the $z$-optimality condition is satisfied:
\begin{equation}\label{ZOptCond}
z^k \in \argmin_z \braces{\norm{Qx^k-z} : z \in Z},
\end{equation} 
we have for each $\omega \in Z^\perp$
\begin{equation}\label{MajBound}
\widehat{\phi}(\omega,x^k,z^k)  \ge \phi^C\left( \omega + \rho(Qx^k-z^k)\right).
\end{equation}
\end{lemma}
\begin{proof}
Via convexity of the term $\norm{Qx-z}_2^2$ over $(x,z) \in \conv(X) \times Z$, we may write the following inequalities that hold for $(x,z) \in \conv(X) \times Z$ and a fixed $\omega \in Z^\perp$:
\begin{align}
L_\rho(x,z,\omega) &\ge f(x) + \omega^\top Qx + \frac{\rho}{2}\norm{Qx^k-z^k}_2^2 \nonumber \\
&\quad\quad\quad+ \rho(Qx^k-z^k)^\top(Qx-z) - \rho(Qx^k-z^k)^\top(Qx^k-z^k) \nonumber\\
&= f(x) + \omega^\top Qx - \frac{\rho}{2}\norm{Qx^k-z^k}_2^2 + \rho(Qx^k-z^k)^\top(Qx-z) \nonumber\\
\Longrightarrow  L_\rho(x,z,\omega) &+ \frac{\rho}{2}\norm{Qx^k-z^k}_2^2 \ge f(x) + \left[\omega+\rho(Qx^k-z^k)\right]^\top Qx \label{MajBoundEq0}\\
&\ge \min_x \braces{ f(x) + \left( \omega + \rho(Qx^k-z^k)\right)^\top Qx : x \in \conv(X)}. \label{MajBoundEq}
\end{align}
Note that the term $-\rho(Qx^k-z^k) z$ vanishes due to the optimality condition associated with~\eqref{ZOptCond}.
Inequality~\eqref{MajBound} follows from the inequalities~\eqref{MajBoundEq0}--\eqref{MajBoundEq} once the substitution $(x,z)=(x^k,z^k)$ and the definition of $\widehat{\phi}(\omega,x^k,z^k)$ are applied to the left-hand side of~\eqref{MajBoundEq0}.
\end{proof}

The convex hull $\conv(X)$ is not known explicitly, and so $\phi^C$ cannot be evaluated directly. 
Consequently, we additionally make use of the following minorization $\widecheck{\phi}$ of $\phi^C$ that can be evaluated.
For $x^k \in \conv(X)$, $k \ge 0$, define $\widecheck{\phi}(\omega,x^k)$ as follows:
\begin{equation}\label{PhiCheck}
\widecheck{\phi}(\omega,x^k) := \min_x \braces{f(x^k) + \nabla_x f(x^k)(x-x^k) + \omega^\top Qx : x \in X}.
\end{equation}
Observe that, due to the linearity of the objective function with respect to $x$ in~\eqref{PhiCheck}, the use of constraint sets $X$ and $\conv(X)$ are interchangeable, 
and so in evaluating $\widecheck{\phi}$, an explicit description of  $\conv(X)$ is not required. 
Furthermore, from the definition of $\phi^C$, the convexity of $f$ over $\mathbb{R}^n$, and the interchangeability of $X$ and $\conv(X)$ in~\eqref{PhiCheck}, it is clear that for all $x^k \in \mathbb{R}^n$, $k \ge 0$, we have $\phi^C(\omega) \ge \widecheck{\phi}(\omega,x^k)$. Furthermore, when $f$ is linear, we have $\phi^C(\omega) \equiv \widecheck{\phi}(\omega,x^k)$ for all $x^k$, $k \ge 0$; the two functions collapse into the same function with the centering at $x^k$ of the latter function now  irrelevant. 


The first important property of $(\omega,x) \mapsto \widecheck{\phi}(\omega,x)$ is its continuity. 
\begin{lemma}\label{PhiCheckLemma1}
Let $X$ be compact, and $f$ be continuously differentiable. Then $(\overline{\omega},\overline{x}) \mapsto \widecheck{\phi}(\overline{\omega},\overline{x})$ is continuous over $(\overline{\omega},\overline{x}) \in  Z^\perp \times \mathbb{R}^n $.
\end{lemma}
\begin{proof}
From~\eqref{PhiCheck}, compute
\begin{align*}
\widecheck{\phi}(\overline{\omega},\overline{x}) &= f(\overline{x}) -\nabla_x f(\overline{x})\overline{x} + \min_x \braces{ \left[ \nabla_x  f(\overline{x}) + \overline{\omega}^\top Q \right] x  + \delta_{\conv(X)}(x)} \\
&= f(\overline{x}) -\nabla_x f(\overline{x})\overline{x} - \delta_{\conv(X)}^*\left(-\left[ \nabla_x  f(\overline{x}) + \overline{\omega}^\top Q \right] \right).
\end{align*}
where $\delta_{\conv(X)}(x):=\left\{\begin{array}{ll}0 & \text{if}\; x \in \conv(X)\\ \infty & \text{otherwise} \end{array} \right.$ is the indicator function on the set $\conv(X)$ and $\delta_{\conv(X)}^*$ is the conjugate function~\cite{Rockafellar1970} of $\delta_{\conv(X)}$.
As $\conv(X)$ is convex and compact, we see that $\delta_{\conv(X)}^*(\cdot)$ has domain $\mathbb{R}^n$ and is thus continuous over $\mathbb{R}^n$ (e.g., Lemma 2.91 of~\cite{Ruszczynski2006}), yielding the intended conclusion.
\end{proof}

The second property of $\widecheck{\phi}$ is its limiting behavior as the solutions $(x^k,z^k)$ approach certain critical values.

\begin{lemma}\label{PhiCheckLemma2}
Let the sequence $\braces{(x^k,z^k)} \subset \conv(X) \times Z$ satisfy the $z$-optimality condition~\eqref{ZOptCond} for each $k \ge 1$. 
If, for some fixed $\omega \in Z^\perp$, the sequence  $\braces{(x^k,z^k)}$ converges optimally in the sense that 
$$\lim_{k \to \infty} (x^k,z^k) = (x^*,z^*) \in \argmin_{x,z} \braces{L_\rho(x,z,\omega) : x \in \conv(X), z \in Z},$$ then 
\begin{equation}\label{PhiCheckConv}
\lim_{k \to \infty} \widecheck{\phi}(\omega + \rho(Qx^k-z^k), x^k) = L_\rho(x^*,z^*,\omega) + \frac{\rho}{2} \norm{Qx^*-z^*}_2^2.
\end{equation}
\end{lemma}
\begin{proof}
We begin by writing the necessary (and sufficient) conditions associated with the optimality $(x^*,z^*) \in \argmin_{x,z} \braces{L_\rho(x,z,\omega) : x \in \conv(X), z \in Z}$:
$$
\left[\begin{array}{c} \nabla f(x^*) + [\omega + \rho(Qx^*-z^*)]^\top Q \\ -\rho (Qx^*-z^*) \end{array}\right] \left[\begin{array}{c} x-x^* \\ z - z^* \end{array}\right] \ge 0 \quad \text{for all}\; x \in \conv(X), z \in Z.
$$
Since $z^k \in \argmin_z \braces{ \norm{Qx^k-z} : z \in Z }$ for each $k \ge 1$, we have $Qx^k - z^k \in Z^\perp$, and so $Qx^* - z^* \in Z^\perp$ also. Thus, we can simplify the consideration of the above displayed necessary conditions to consider the $x$ block only:
$$
\left[\begin{array}{c} \nabla f(x^*) + [\omega + \rho(Qx^*-z^*)]^\top Q  \end{array}\right] \left[\begin{array}{c} x-x^* \end{array}\right] \ge 0 \quad \text{for all}\; x \in \conv(X),
$$
which implies
$$
\min_x \braces{\left[\begin{array}{c} \nabla f(x^*) + [\omega + \rho(Qx^*-z^*)]^\top Q  \end{array}\right] \left[\begin{array}{c} x-x^* \end{array}\right] : x \in \conv(X)} = 0.
$$
In terms of $\widecheck{\phi}(\omega + \rho(Qx^*-z^*), x^*)$, the above equality is re-written as:
\begin{align*}
\widecheck{\phi}(\omega + \rho(Qx^*-z^*), x^*) &= f(x^*) + (\omega)^\top Qx^* + \rho \norm{Qx^* - z^*}_2^2 \\
&= L_\rho(x^*,z^*,\omega) + \frac{\rho}{2} \norm{Qx^* - z^*}_2^2,
\end{align*}
where the equality $(Qx^*-z^*)^\top z^* = 0$ is utilized.
The continuity of $(\overline{\omega},\overline{x}) \mapsto \widecheck{\phi}(\overline{\omega},\overline{x})$ established in Lemma~\ref{PhiCheckLemma1} gives the desired conclusion.
\end{proof}

We use Lemmas~\ref{PhiCheckLemma1} and~\ref{PhiCheckLemma2} to develop a proximal bundle method-like serious step condition (SSC) that makes use of $\widehat{\phi}$ and $\widecheck{\phi}$ in place of the cutting plane model and $\phi$, respectively. Defining
$\widetilde{\omega}^k := \omega^k + \rho(Qx^k-z^k)$, consider the following modified serious step condition:
\begin{equation}\label{SSCnew}
\gamma \le \frac{\widecheck{\phi}(\widetilde{\omega}^k,x^k)-\widecheck{\phi}(\omega^k,x^{k-1})}{\widehat{\phi}(\omega^k,x^k,z^k)-\widecheck{\phi}(\omega^k,x^{k-1})} \le 1, 
\end{equation}
where $\gamma \in (0,1)$ is the SSC parameter. The upper bound of~\eqref{SSCnew} is satisfied automatically since $\widehat{\phi}(\omega^k,x^k,z^k) \ge \phi^C(\widetilde{\omega}^k) \ge \widecheck{\phi}(\widetilde{\omega}^k,x^k)$ holds by Lemma~\ref{LemmaMajBound} and the definition of $\widecheck{\phi}$. However, the satisfaction of the lower bound is conditional on $\gamma$.

\begin{remark}
Throughout this paper, we shall always assume or construct $z^k$ such that the $z$-optimality condition~\eqref{ZOptCond} is satisfied for each $k \ge 0$.
Due to the necessary conditions of optimality associated with~\eqref{ZOptCond} and that $Z$ is a linear subspace,
we have $(Q{x}^k-{z}^k)^\top z = 0$ for all $z \in Z$. It immediately follows that if $\omega^k \in Z^\perp$, then $\widetilde{\omega}^k=\omega^k+\rho(Qx^k-z^k) \in Z^\perp$ also.
Thus, the satisfaction of the $z$-optimality condition~\eqref{ZOptCond} 
guides the generation of $\braces{\omega^k}$ so that if $\omega^0 \in Z^\perp$, then $\omega^k \in Z^\perp$ is always maintained for each $k \ge 1$. 
\end{remark}

Under certain circumstances, the denominator of the ratio displayed in~\eqref{SSCnew} can be zero. The following lemma states that this never happens when $\omega^k$ is \emph{not} dual optimal with respect to the dual problem~\eqref{P2C}.
\begin{lemma}\label{LemmaDenNotZero}
For any $\omega \in Z^\perp$ that is not dual optimal with respect to the dual problem~\eqref{P2C} and $(x,z) \in \conv(X) \times Z$, we have 
\begin{equation}\label{EqDenNotZero}
\widehat{\phi}(\omega,x,z)-\phi^C(\omega) > 0.
\end{equation}
Consequently, at any iteration $k$, the denominator of the ratio displayed in~\eqref{SSCnew} cannot be zero when $\omega^k$ is not dual optimal.
\end{lemma}
\begin{proof}
By the definition of $\widehat{\phi}$ , we have
\begin{align*}
\widehat{\phi}(\omega,x,z)-\phi^C(\omega) & \ge L_\rho(x^*,z^*,\omega) +\frac{\rho}{2}\norm{Qx - z}_2^2 - \phi^C(\omega), 
\end{align*}
where $(x^*,z^*) \in \argmin_{x,z} \braces{L_\rho(x,z,\omega) : x \in \conv(X), z \in Z}$. (That is, we substitute $L_\rho(x,z,\omega)$ from the definition of $\widehat{\phi}$ with $L_\rho(x^*,z^*,\omega)$ to get the inequality.)
Now $L_\rho(x^*,z^*,\omega)- \phi^C(\omega) >0$ when $\omega$ is not dual optimal. Otherwise, if $L_\rho(x^*,z^*,\omega) = \phi^C(\omega)$, then $Qx^*=z^*$ must hold, and $(x^*,z^*,\omega)$ is a Lagrangian saddle point for problem~\eqref{P1ConvB} with respect to the Lagrangian relaxation of the constraint $Qx=z$.
This contradicts the non-dual optimality of $\omega$. Thus, the strict inequality~\eqref{EqDenNotZero} is established.

In the context of~\eqref{SSCnew} at iteration $k$, noting that $\phi^C(\omega^k) \ge \widecheck{\phi}(\omega^k,x^{k-1})$, we substitute $(x,z)=(x^k,z^k)$ and $\omega=\omega^k$ in the strict inequality~\eqref{EqDenNotZero} and so the denominator in~\eqref{SSCnew} is positive when $\omega^k$ is not dual optimal. 
\end{proof}

From Lemma~\ref{PhiCheckLemma2}, we have the following result regarding the satisfaction of condition~\eqref{SSCnew}.
\begin{proposition}\label{PropSSC}
Let the sequence 
$\braces{(x^k,z^k)} \subset \conv(X) \times Z$ satisfy 
$$z^k \in \argmin_z \braces{ \norm{Qx^k-z} : z \in Z }$$ for each $k \ge 1$.
Furthermore, let $\omega \in Z^\perp$ and $\omega \not\in \argmax_\omega \phi(\omega)$.
If the sequence  $\braces{(x^k,z^k)}$ converges optimally in the sense that 
$$\lim_{k \to \infty} (x^k,z^k) = (x^*,z^*) \in \argmin_{x,z} \braces{L_\rho(x,z,\omega) : x \in \conv(X), z \in Z},$$ 
then condition~\eqref{SSCnew} must be satisfied after a finite number of iterations.
\end{proposition}
\begin{proof}
For all $(x^k,z^k) \in \conv(X) \times Z$ with $z^k \in \argmin_z \braces{\norm{Qx^k-z}_2^2}$, we have 
\begin{align*}
\widehat{\phi}(\omega,x^k,z^k)&= L_\rho(x^k,z^k,\omega) + \frac{\rho}{2}\norm{Qx^k-z^k}_2^2 \\
&\ge \phi^C(\omega+\rho(Qx^k-z^k)) \ge \widecheck{\phi}(\omega+\rho(Qx^k-z^k),x^k),
\end{align*}
where the first inequality follows from the definition of $\widehat{\phi}$ and Lemma~\ref{LemmaMajBound}, and the second inequality follows readily from the definition of $\widecheck{\phi}$.
By the assumption that $\omega$ is not dual optimal, the denominator of~\eqref{SSCnew} cannot be zero by Lemma~\ref{LemmaDenNotZero}. 
It follows from the convergence in~\eqref{PhiCheckConv} implied by Lemma~\ref{PhiCheckLemma2} that the ratio in~\eqref{SSCnew} must approach 1, and so condition~\eqref{SSCnew} must be satisfied after a finite number of iterations.
\end{proof}
Consequently, unless the current $\omega^k$ is already dual optimal, there cannot be an infinite number of null-steps when using condition~\eqref{SSCnew}.

Algorithm~\ref{AlgDualAscent} provides a general framework for an AL method with approximate subproblem solutions. 
The inputs $f$, $Q$, $X$, and $Z$ specify the data associated with problem~\eqref{P1}; $\rho>0$ is the AL term coefficient; $\omega^0$ is an initial dual solution; $\gamma \in (0,1)$ is the parameter of the serious step condition~\eqref{SSCnew}; and $\epsilon > 0$ is a tolerance for termination. Algorithm~\ref{AlgDualAscent} will be given a specific implementation in the form of SDM-GS-ALM in Section~\ref{Sect3}. The convergence proof of Algorithm~\ref{AlgDualAscent} is based on standard ideas in the convergence proofs of the proximal bundle method such as found in Chapter 7 of~\cite{Ruszczynski2006}. 
\begin{algorithm}[H]
\caption{A general approximated ALM using a bundle method SSC. \label{AlgDualAscent}}
\begin{algorithmic}[1]
\State {\bf Preconditions:} $\omega^1 \in Z^\perp$, $\gamma \in (0,1)$.
\Function{ApproxALM}{$f$, $Q$, $X$, $Z$, $\rho$, $\omega^1$, $\gamma$, $\epsilon$, $k_{max}$}  
       \For{$k=1,2,\dots,k_{max}$}
       \State Solve approximately
       \State \quad$(x^k,z^k) \in \argmin_{x,z} \braces{L_\rho(x,z,\omega^k) : x \in \conv(X), z \in Z}$ \label{LineSP1a}
       such that 
       \State \quad\quad 1) $z ^k \in \argmin_z \braces{\norm{Qx^k - z}_2^2 : z \in Z}$ and \label{LineSP1b}
       \State \quad\quad 2) either  
       \State \quad\quad \quad\quad $\widehat{\phi}(\omega^k,x^k,z^k)-\widecheck{\phi}(\omega^k,x^{k-1}) \le \epsilon$ or \label{Alg2LineDenNearZero}
       \State \quad\quad \quad\quad $0 < \gamma \le \frac{\widecheck{\phi}\left(\omega^k + \rho(Qx^k-z^k),x^k\right)-\widecheck{\phi}(\omega^k,x^{k-1})}{\widehat{\phi}(\omega^k,x^k,z^k)-\widecheck{\phi}(\omega^k,x^{k-1})}$ \label{Alg2LineSSCholds}
       \If{$\widehat{\phi}(\omega^k,x^k,z^k)- \widecheck{\phi}(\omega^k,x^{k-1}) \le \epsilon$}
       	\State {\bf return} $(x^k,z^k,\omega^k)$
       \Else
         \State set $\omega^{k+1} \gets \omega^{k} + \rho (Qx^k-z^k)$
       \EndIf
       	
        \EndFor  
      \State \textbf{return}  $(x^k,z^k,\omega^{k+1})$
\EndFunction
\end{algorithmic}
\end{algorithm} 
\begin{proposition}\label{DualConvergence}
Assume that problem~\eqref{P2C} has an optimal dual solution $\omega^*$, and that for each $k \ge 1$, $\phi^C(\omega^k) < \phi^C(\omega^*)$.
If the sequence $\braces{\omega^k}$ of dual updates is generated with Algorithm~\ref{AlgDualAscent} with $\epsilon=0$ and $k_{max}=\infty$, then $\braces{\omega^k}$ converges, and 
$\lim_{k \to \infty} \widecheck{\phi}(\omega^k,x^{k-1}) = \zeta^{CLD}$ (and consequently $\lim_{k \to \infty} \phi^C(\omega^k) = \zeta^{CLD}$). 
Furthermore, $$\lim_{k \to \infty} \widehat{\phi}(\omega^k,x^k,z^k) = \zeta^{CLD},$$ and all limit points $(\bar{x},\bar{z})$ of the sequence $\braces{(x^k,z^k)}$ are optimal for problem~\eqref{P1ConvB}.
\end{proposition}
\begin{proof}
Let $\omega^*$ be any dual optimal solution for problem~\eqref{P2C}.
For each iteration $k \ge 1$, write the following two relations:
\begin{align}
 \norm{\omega^{k+1}-\omega^{*}}_2^2  = & \norm{\omega^{k}-\omega^{*} + \rho(Qx^k-z^k)}_2^2 \nonumber \\
  = &\norm{\omega^{k}-\omega^{*}}_2^2  + 2\rho (Qx^k-z^k)^\top (\omega^k-\omega^*) + \rho^2 \norm{Qx^k-z^k}_2^2, \label{R1} \\
 \text{and } \quad \phi^C(\omega^*) 
\le & L_\rho(x^k,z^k,\omega^*) = L_\rho(x^k,z^k,\omega^k) + (\omega^*-\omega^k)^\top (Qx^k-z^k) \nonumber \\
\Longrightarrow \;& (\omega^k - \omega^*)^\top(Qx^k-z^k) \le L_\rho(x^k,z^k,\omega^k) - \phi^C(\omega^*).\label{R2}
\end{align}
Substituting the inequality~\eqref{R2} into equality~\eqref{R1}, we have
\begin{align}
\norm{\omega^{k+1}-\omega^{*}}_2^2 &\le \norm{\omega^{k}-\omega^{*}}_2^2 \nonumber \\
& \quad + 2\rho \left[ L_\rho(x^k,z^k,\omega^k) - \phi^C(\omega^*) \right] + \rho^2 \norm{Qx^k-z^k}_2^2 \\ 
&= \norm{\omega^{k}-\omega^{*}}_2^2 + 2\rho \left[ \widecheck{\phi}(\omega^k,x^{k-1}) - \phi^C(\omega^*) \right] \nonumber \\
&\quad + 2\rho \left[ L_\rho(x^k,z^k,\omega^k) +\frac{\rho}{2} \norm{Qx^k-z^k}_2^2- \widecheck{\phi}(\omega^k,x^{k-1}) \right]. \label{R3}
\end{align}
By assumption, for each $k \ge 1$, we have $\phi^C(\omega^k) < \phi^C(\omega^*)$, so by Lemma~\ref{LemmaDenNotZero} and $\epsilon=0$, the Line~\ref{Alg2LineDenNearZero} condition of Algorithm~\ref{AlgDualAscent} never holds. Thus, the Line~\ref{Alg2LineSSCholds} condition, which is equivalent to the satisfaction of condition~\eqref{SSCnew}, is satisfied for each $k \ge 1$. 
Rewriting~\eqref{SSCnew}, with the substitution $\widetilde{\omega}^k = \omega^{k+1}$, as
\begin{equation}\label{R4}
L_\rho(x^k,z^k,\omega^k) + \frac{\rho}{2} \norm{Qx^k-z^k}_2^2-\widecheck{\phi}(\omega^k,x^{k-1}) \le \frac{\widecheck{\phi}(\omega^{k+1},x^{k})-\widecheck{\phi}(\omega^k,x^{k-1})}{\gamma}
\end{equation}
and substituting~\eqref{R4} into~\eqref{R3}, we have
\begin{align}
\norm{\omega^{k+1}-\omega^{*}}_2^2 \le \norm{\omega^{k}-\omega^{*}}_2^2 &+ 2\rho \left[ \widecheck{\phi}(\omega^k,x^{k-1}) - \phi^C(\omega^*) \right] \nonumber\\
&+ \frac{2\rho}{\gamma} \left[ \widecheck{\phi}(\omega^{k+1},x^{k})-\widecheck{\phi}(\omega^k,x^{k-1}) \right]. \label{R5}
\end{align}
From~\eqref{R5}, we make the following three inferences: 1) that $\braces{\norm{\omega^k-\omega^*}}$ is bounded, 2) that $\sum_{k=1}^\infty \left[ \phi^C(\omega^*) - \phi^C(\omega^k) \right]$ is finite, and 3) that $\braces{\omega^k}$ converges. To establish these inferences, we sum the inequality~\eqref{R5} from $k=\ell,\dots,N$ for some integers $1 \le \ell \le N$
to get
\begin{align}
&2\rho \sum_{k=\ell}^N \left[ \phi^C(\omega^*) - \widecheck{\phi}(\omega^k,x^{k-1}) \right]  + \norm{\omega^{N+1}-\omega^{*}}_2^2 \nonumber\\
&\quad\quad\quad\quad\quad\quad\quad\quad \le \norm{\omega^{\ell}-\omega^{*}}_2^2  +  \frac{2\rho}{\gamma} \left[ \widecheck{\phi}(\omega^{N+1},x^{N})-\widecheck{\phi}(\omega^\ell,x^{\ell-1}) \right] \nonumber \\
\Longrightarrow \;\; &2\rho \sum_{k=\ell}^N \left[ \phi^C(\omega^*) - \widecheck{\phi}(\omega^k,x^{k-1}) \right]  + \norm{\omega^{N+1}-\omega^{*}}_2^2 \nonumber\\
&\quad\quad\quad\quad\quad\quad\quad\quad \le \norm{\omega^{\ell}-\omega^{*}}_2^2  +  \frac{2\rho}{\gamma} \left[ \phi^C(\omega^{*})-\widecheck{\phi}(\omega^\ell,x^{\ell-1}) \right] \label{R6}
\end{align}
where the last inequality is straightforward due to $\widecheck{\phi}(\omega^{N+1},x^{N}) \le \phi^C(\omega^{N+1}) \le \phi^C(\omega^*)$ implied by the optimality of $\omega^*$. Noting that each summand $\phi^C(\omega^*) - \widecheck{\phi}(\omega^k,x^{k-1})$ in the summation on the left-hand side of~\eqref{R6} is nonnegative, we have immediately from~\eqref{R6} that $\sum_{k=1}^\infty \left[ \phi^C(\omega^*) - \widecheck{\phi}(\omega^k,x^{k-1}) \right] < \infty$ and $\braces{(\omega^k - \omega^*)}$ is bounded, establishing the first two inferences from~\eqref{R5}. 
The validity of the first two inferences imply the boundedness of $\braces{\omega^k}$ and the convergence $\lim_{k \to \infty} \widecheck{\phi}(\omega^k,x^{k-1}) = \phi^C(\omega^*)$, respectively. The boundedness of $\braces{\omega^k}$ implies the existence of limit points, while the convergence $\lim_{k \to \infty} \widecheck{\phi}(\omega^k,x^{k-1}) = \phi^C(\omega^*)$ implies that all such limit points are dual optimal. It is straightforward from the bounding relationships 
$$\phi^C(\omega^*) > \phi^C(\omega^k) \ge \widecheck{\phi}(\omega^k,x^{k-1})$$ 
that $\lim_{k \to \infty} \phi^C(\omega^k) = \phi^C(\omega^*)$ also.

To establish the third assertion, that $\braces{\omega^k}$ in fact converges, we drop the summation from the left-hand side of~\eqref{R6},
\begin{equation}\label{R7}
\norm{\omega^{N+1}-\omega^{*}}_2^2 \le \norm{\omega^{\ell}-\omega^{*}}_2^2  +  \frac{2\rho}{\gamma} \left[ \phi^C(\omega^{*})-\widecheck{\phi}(\omega^\ell,x^{\ell-1}) \right],
\end{equation}
and note that the above analysis holds independent of the choice of dual optimal $\omega^*$. Since it was just shown that $\braces{\omega^k}$ has limit points, and that all such limit points are dual optimal, we now specify $\omega^*$ to be one of these limit points.  We then choose an appropriate $\ell$ for any $\varepsilon > 0$ so that the right-hand side of~\eqref{R7} is arbitrarily small, i.e.,
$$
\norm{\omega^{N+1}-\omega^{*}}_2^2 \le \varepsilon
$$
for all $N \ge \ell$.
Thus, $\lim_{k \to \infty} \omega^k = \omega^*$, and it is clear that the limit point $\omega^*$ of $\braces{\omega^k}$ is in fact unique.

To prove the last assertion, the satisfaction of~\eqref{SSCnew} is rewritten as
\begin{align*}
{\widecheck{\phi}({\omega}^{k+1},x^k)-\widecheck{\phi}(\omega^k,x^{k-1})} &\le {\widehat{\phi}(\omega^k,x^k,z^k)-\widecheck{\phi}(\omega^k,x^{k-1})} \\
 &\le \frac{1}{\gamma}  \left( {\widecheck{\phi}({\omega}^{k+1},x^k)-\widecheck{\phi}(\omega^k,x^{k-1})}\right).
\end{align*}
Due to the convergence $\lim_{k \to \infty} \widecheck{\phi}(\omega^k,x^{k-1}) = \zeta^{CLD}$,
we have on taking the limit as $k \to \infty$ of the last displayed inequalities that $\lim_{k \to \infty} \widehat{\phi}(\omega^k,x^k,z^k) = \zeta^{CLD}.$
In taking the limit points $(\bar{x},\bar{z},\omega^*)$ of the sequence $\braces{(x^k,z^k,\omega^k)}$, noting that the optimal value of problem~\eqref{EqQMPConv} with $\omega=\omega^*$ is $\zeta^{CLD}$ by Lemma~\ref{LemmaALOpt}, we have
$$
\zeta^{CLD} + \frac{\rho}{2}\norm{Q\bar{x}-\bar{z}}_2^2 \le L_\rho(\bar{x},\bar{z},\omega^*) + \frac{\rho}{2}\norm{Q\bar{x}-\bar{z}}_2^2 = \zeta^{CLD}.
$$
From this, it follows that $\norm{Q\bar{x}-\bar{z}}_2^2 = 0$ and $L_\rho(\bar{x},\bar{z},\omega^*) = \zeta^{CLD}$, and so $(\bar{x},\bar{z})$ must be feasible and furthermore optimal for~\eqref{P1ConvB}.
\end{proof}

\section{Main algorithm}\label{Sect3}
After integrating SDM and the nonlinear block Gauss-Seidel method, a practical implementation of Algorithm~\ref{AlgDualAscent} is provided in this section.
  
We consider the following general two-block problem
\begin{equation}\label{P-GS}
\min_{x,z} \braces{F(x,z) : x \in \conv(X), z \in Z}
\end{equation}
where $F : \mathbb{R}^{n} \times \mathbb{R}^{q} \mapsto \mathbb{R}$ is a continuously differentiable function, $\conv(X)$ and $Z$ are closed convex sets, and $\conv(X)$ is also bounded. ($Z$ can be more generally a convex set in this setting, not necessarily a linear (sub)space.) Additionally, we assume for each fixed $x \in \conv(X)$ that $z \mapsto F(x,z)$ is inf-compact. (That is, the set $\braces{z \in Z : F(x,z) \le \ell}$ is compact for all $x \in \conv(X)$ and $\ell \in \mathbb{R}$.)

Problem~\eqref{P-GS} is assumed to be feasible, bounded, and to have an optimal solution $(x^*,z^*)$. We shall utilize the following two-block nonlinear Gauss-Seidel (GS) method with the $x$ update approximated in a manner resembling an iteration of the SDM.

\begin{algorithm}[H]
\caption{An iteration of inner-approximated nonlinear Gauss-Seidel approach applied to problem~\eqref{P-GS}. \label{AlgBCD-XZ}}
\begin{algorithmic}[1]
\State Precondition: $\widetilde{x} \in \conv(X)$, $\widetilde{z} \in \argmin_z \braces{F(\widetilde{x},z) : z \in Z}$, $D \subseteq \conv(X)$
\Function{SDM-GS}{$F$, $X$, $Z$, $D$, $\widetilde{x}$, $\widetilde{z}$, $t_{max}$}  
        \For{$t=1,\dots,t_{max}$} \label{ForXYUpdateBegin}
           \State $\widetilde{x} \gets \argmin_x \braces{F(x,\widetilde{z}) : x \in D}$ \label{SDMGSXUpdate}
	\State $\widetilde{z} \gets \argmin_z \braces{F(\widetilde{x},z) : z \in Z}$ \label{SDMGSZUpdate}
        \EndFor \label{ForXYUpdateEnd}
	\State $\widehat{x} \in \argmin_x \braces{\nabla_x F(\widetilde{x},\widetilde{z}) (x - \widetilde{x}) : x \in X}$ \label{SCGDirFindingSP}
	 \State Reconstruct ${D}$ to be any set such that \label{LineDa}
	 \State \quad $\braces{\widetilde{x} + \alpha (\widehat{x} - \widetilde{x}) : \alpha \in [0,1]} \subseteq {D} \subseteq \conv(X)$ \label{LineDb}
	 \State Set $\Gamma \gets - \nabla_x F(\widetilde{x},\widetilde{z}) (\widehat{x} - \widetilde{x})$ \label{GammaLine}
      \State \textbf{return} $(\widetilde{x},\widetilde{z},{D}, \Gamma )$ 
\EndFunction
\end{algorithmic}
\end{algorithm} 

If the $z$ block update of Line~\ref{SDMGSZUpdate} is trivialized, such as by making it not actually appear in the definition of $F$, or by making $Z$ a singleton set, 
then Algorithm~\ref{AlgBCD-XZ} would be identical to SDM applied to problem~\eqref{P-GS} in which the $z$ block of variables correspondingly does not play any role.
On the other hand, if the $x$ update~\eqref{SDMGSXUpdate} is replaced with an update based on an exact minimization $\widetilde{x} \gets \argmin_x \braces{F(x,\widetilde{z}) : x \in \conv(X)}$ (so that the computations of Lines~\ref{SCGDirFindingSP}--\ref{GammaLine} and the returning of $D$ and $\Gamma$ can be skipped), then Algorithm~\ref{AlgBCD-XZ} would be equivalent to a more traditional two-block nonlinear Gauss-Seidel method. 
Different forms of approximation of the $x$ update, such as those resulting from gradient descent steps in $x$, are also considered in~\cite{HathawayBezdek1991,Bonettini2011}.

\begin{remark}
The main approach envisioned for constructing the inner approximation $D$ on Lines~\ref{LineDa}--\ref{LineDb} 
is to take $D \gets \conv(D \cup \braces{\widetilde{x},\widehat{x}})$. To implement this update of $D$, we need to save the points $\widehat{x}$ computed during previous calls to Algorithm~\ref{AlgBCD-XZ}. 
\end{remark}

We assume in the following proposition that Algorithm~\ref{AlgBCD-XZ} is applied iteratively in the sense that at iteration $k\ge 0$, we input $(\widetilde{x},\widetilde{z})=(x^k,z^k)$ and return $(\widetilde{x},\widetilde{z})=(x^{k+1},z^{k+1})$. Furthermore, at the same iteration $k$ call of Algorithm~\ref{AlgBCD-XZ}, we set $d^{k+1} = \widehat{x} - \widetilde{x}$ where $\widehat{x}$ and $\widetilde{x}$ are set as in Line~\ref{LineDb}. This provides a reference sequence of directions $\braces{d^k}$ necessary in the proof of the following proposition.
\begin{proposition}\label{PropGSSDM}
For problem~\eqref{P-GS}, let $F$ be convex and continuously differentiable, and let $\conv(X)$ and $Z$ be nonempty and convex, with $\conv(X)$ bounded and $z \mapsto F(x,z)$ inf-compact for each $x \in\conv(X)$. Then, for any $t_{max} \ge 1$, the sequence $\braces{(x^k,z^k)}$ generated by iterations of Algorithm~\ref{AlgBCD-XZ} has limit points $(\bar{x},\bar{z})$, each of which are optimal for problem~\eqref{P-GS}.
\end{proposition}
\begin{proof}
In light of the convexity and continuous differentiablity of $F$ and the convexity of $\conv(X)$ and $Z$, it is sufficient to show that 
\begin{align}\label{SDMGSXStat}
\nabla_{x} F(\bar{x},\bar{z}) (x-\bar{x}) &\ge 0\quad \text{for all}\; x \in \conv(X) \\
\text{and} \quad 
\nabla_{z} F(\bar{x},\bar{z}) (z-\bar{z}) &\ge 0\quad \text{for all}\;z \in Z.\label{SDMGSZStat}
\end{align}
As $\nabla_{z} F(x^k,z^k) (z-z^k) \ge 0$ for all $z \in Z$  holds for each $k \ge 1$ (this follows due to the optimality $z^k \in \argmin_z \braces{F(x^k,z) : z \in Z}$  that holds by construction) the satisfaction of the latter condition~\eqref{SDMGSZStat} is trivially established for any limit points $(\bar{x},\bar{z})$. It remains only to show the satisfaction of the $x$-stationarity condition~\eqref{SDMGSXStat}. This may be established by using Proposition 3.2 of~\cite{Bonettini2011} combined with the last sentence of Remark 3.3 from the same reference. 
But for the sake of explicitness, we use developments in Appendix~\ref{AppSDMGS} to show that~\eqref{SDMGSXStat} holds.

Note, for the sake of nontriviality, that $\nabla_{x} F(x^k,z^k) (x-x^k) \ge 0$ for all $x \in X$ is assumed \emph{not} to hold for any $k \ge 1$. 
Thus, for the reference sequence of directions $\braces{d^k}$ mentioned immediately before the statement of the proposition,
the Direction Assumption (DA) referred to in Appendix~\ref{AppSDMGS} holds.  Also, the Gradient Related Assumption (GRA) referred to in Appendix~\ref{AppSDMGS} is satisfied for this same $\braces{d^k}$ by Lemma~\ref{LemmaDir} therein. Due to the construction of $D$ in Line~\ref{LineDb} and setting $(x^{k+1},z^{k+1}) = (\widetilde{x},\widetilde{z})$ after the termination of the for loop of Lines~\ref{ForXYUpdateBegin}--\ref{ForXYUpdateEnd}, we have given $\braces{d^k}$ and any choice of $(\beta,\sigma) \in (0,1)$ the satisfaction of the Sufficient Decrease Assumption (SDA) referred to in Appendix~\ref{AppSDMGS}. It then follows from Lemma~\ref{SCGXStatLemma} of Appendix~\ref{AppSDMGS}  that limit points $(\bar{x},\bar{z})$ of $\braces{(x^k,z^k)}$ do exists, each of which satisfy the stationarity condition~\eqref{SDMGSXStat}.    
\end{proof}

The method SDM-GS-ALM is now stated as Algorithm~\ref{AlgDualAscentGSSSC}, which uses Algorithm~\ref{AlgBCD-XZ} as a subroutine to provide a practical implementation of Algorithm~\ref{AlgDualAscent}
\begin{remark}
At the return of Algorithm~\ref{AlgBCD-XZ} in Line~\ref{ReturnSDMGS} of Algorithm~\ref{AlgDualAscentGSSSC}, we have
\begin{align*}
\Gamma &= -\nabla_x L_\rho({x}^{k},{z} ^{k},\omega^k) (\widehat{x} - x^k) \\&= -\left[ \nabla_x f(x^k) + \left( \omega^k + \rho(Qx^k-z^k)\right)^\top Q\right](\widehat{x}-x^k )
\end{align*}
where $\widehat{x}$ is computed on Line~\ref{SCGDirFindingSP} of Algorithm~\ref{AlgBCD-XZ}.
One may verify using this value of $\Gamma$, the equality $(Qx^k-z^k)^\top z^k = 0$ due to $z^k \in \argmin_z \braces{\norm{Qx^k-z}_2^2 : z \in Z}$, and the computation of $\widetilde{\phi}$ on Lines~\ref{ComputePhiTilde0} and~\ref{ComputePhiTilde} that for $k \ge 0$, 
$$\widetilde{\phi} =  L_{\rho}({x}^{k},{z} ^{k},\omega^k) + \frac{\rho}{2} \norm{Q{x}^{k}-{z}^{k}}_2^2 - \Gamma = \widecheck{\phi}\left(\omega^{k} + \rho (Q{x}^{k}-{z}^{k}),x^k \right).$$ 
\end{remark}

\begin{algorithm}[H]
\caption{A practical implementation of Algorithm~\ref{AlgDualAscent} based on the use of SDM-GS iterations. (SDM-GS is given as Algorithm~\ref{AlgBCD-XZ}.)\label{AlgDualAscentGSSSC}}
\begin{algorithmic}[1]
\State {\bf Preconditions:} $x^0 \in \conv(X)$, $z^0 \in Z$, $\omega^0 \in Z^\perp$, $D \subseteq \conv(X)$, $\gamma \in (0,1)$.
\Function{SDM-GS-ALM}{$f$, $Q$, $X$, $Z$, $D$, $\rho$, $x^0$, $z^0$, $\omega^0$, $\gamma$, $\epsilon$, $t_{max}$, $k_{max}$}  
       \State $(x^0,z^0,D,\Gamma) \gets$ SDM-GS($L_\rho(\cdot,\cdot,\omega^0)$, $X$, $Z$, $D$, $x^0$, $z^0$, $t_{max}$)
       \State $\widetilde{\phi} \gets  L_{\rho}({x}^{0},{z} ^{0},\omega^0) + \frac{\rho}{2} \norm{Q{x}^{0}-{z}^{0}}_2^2 - \Gamma$ \label{ComputePhiTilde0}
       \State set $\omega^{0} \gets \omega^{0} + \rho (Q{x}^{0}-{z}^{0})$, $\widecheck{\phi}^{0} \gets \widetilde{\phi}$
       \For{$k=1,2,\dots,k_{max}$}
       \State Initialize $\omega^k \gets \omega^{k-1}$, $\widecheck{\phi}^k \gets \widecheck{\phi}^{k-1}$ \Comment{(Default, null-step updates)}
       \State $(x^k,z^k,D,\Gamma) \gets$ SDM-GS($L_\rho(\cdot,\cdot,\omega^k)$, $X$, $Z$, $D$, $x^{k-1}$, $z^{k-1}$, $t_{max}$) \label{ReturnSDMGS}
       \If{$L_\rho({x}^{k},{z} ^{k},\omega^k) + \frac{\rho}{2} \norm{Q{x}^{k}-{z}^{k}}_2^2-\widecheck{\phi}^k \le \epsilon$} \label{Alg4BeginTermCond}
       	\State {\bf return} $(x^k,z^k,\omega^k,\widecheck{\phi}^k)$
       \EndIf \label{Alg4EndTermCond}
       \State $\widetilde{\phi} \gets  L_{\rho}({x}^{k},{z} ^{k},\omega^k) + \frac{\rho}{2} \norm{Q{x}^{k}-{z}^{k}}_2^2 - \Gamma$ \label{ComputePhiTilde}
       \State $\gamma^k \gets \frac{\widetilde{\phi} - \widecheck{\phi}^k}{L_\rho({x}^{k},{z}^{k}, \omega^k) + \frac{\rho}{2} \norm{Q{x}^{k}-{z} ^{k}}_2^2-\widecheck{\phi}^k}$ \label{ComputeCritVal}
       \If{$\gamma^k \ge \gamma$} \label{Alg4LineSSC}
         \State set $\omega^{k} \gets \omega^{k} + \rho (Q{x}^{k}-{z}^{k})$, $\widecheck{\phi}^{k} \gets \widetilde{\phi}$
        \EndIf
        \State Possibly update $\rho$, e.g., $\rho \gets \frac{1}{\min \braces{\max \braces{(2/\rho)(1-\gamma^k),1/(10\rho),10^{-4}}, 10/\rho}}$ as in~\cite{Kiwiel1995} \label{Alg4RhoUpdate}
        \EndFor  
      \State \textbf{return}  $(x^k,z^k,\omega^{k},\widecheck{\phi}^k)$
\EndFunction
\end{algorithmic}
\end{algorithm}

\begin{proposition}\label{PropFWBCD}
Let $\braces{(x^k,z^k,\omega^k)}$ be a sequence generated by Algorithm~\ref{AlgDualAscentGSSSC} applied to problem~\eqref{P1} with $X$ compact, $Z$ a linear subspace, $\omega^0 \in Z^{\perp}$, $\rho > 0$, $\gamma \in (0,1)$, $\epsilon=0$ and $k_{max}=\infty$. 
If there exists a dual optimal solution $\omega^*$ to the dual problem~\eqref{P2C}, then either
\begin{enumerate}
\item $\omega^k = \overline{\omega}$ is fixed and optimal for~\eqref{P2C} for $k \ge \bar{k}$ for some finite $\bar{k}$; or
\item $\omega^k$ is never optimal for~\eqref{P2C} for any finite $k \ge 1$, but $\lim_{k \to \infty} \omega^k = \overline{\omega}$ is optimal,
\end{enumerate}
and the sequence $\braces{(x^k,z^k)}$ has limit points $(\overline{x},\overline{z})$, each of which are optimal for problem~\eqref{P1ConvB}.
\end{proposition}
\begin{proof}
In the first case, Algorithm~\ref{AlgDualAscentGSSSC} never takes serious steps for iterations $k \ge \bar{k} \ge 1$, and so with $\omega^k=\overline{\omega}$ fixed for $k \ge \bar{k}$, 
Algorithm~\ref{AlgDualAscentGSSSC} iterations continue with the generation of $\braces{(x^k,z^k)}$ as generated by iterations of SDM-GS (Algorithm~\ref{AlgBCD-XZ}). 
By Proposition~\ref{PropGSSDM}, the sequence $\braces{(x^k,z^k)}$ has limit points $(\overline{x},\overline{z})$, each of which are optimal for problem~\eqref{EqQMPConv} with $\omega=\overline{\omega}$.
Then, by Lemma~\ref{LemmaALOpt}, $(\overline{x},\overline{z})$ is also optimal for problem~\eqref{P1ConvB} since $\overline{\omega}$ is optimal for~\eqref{P2C}.

In the second case where $\omega^k$ is never dual optimal for~\eqref{P1ConvB} for any finite $k \ge 1$, any serious step must be followed by a finite number of consecutive null-steps.
We consider the subsequence indices $\braces{k_i}_{i=1}^\infty$ where the update $\omega^{k_i+1}$ is obtained by a serious step. By Proposition~\ref{DualConvergence},
we have $\lim_{i \to \infty} \phi^C(\omega^{k_i+1}) =  \zeta^{CLD}$, and taking into account the null steps in between, we have also $\lim_{k} \phi^C(\omega^{k}) =  \zeta^{CLD}$. 
%
%
To prove the last claim, 
we note that $\omega^{j}=\omega^{k_{i+1}}$ for all integers $j$ such that $k_i < j \le k_{i+1}$ due to the taking of null steps.
From Proposition~\ref{DualConvergence}, we have that $\lim_{i \to \infty} L_\rho(x^{k_i},z^{k_i},\omega^{k_i}) = \zeta^{CLD}$.
By the continuity of $(x,z,\omega) \mapsto L_\rho(x,z,\omega)$, the convergences $\lim_{k \to \infty} \omega^k = \overline{\omega}$ and $\lim_{i \to \infty} Qx^{k_i}-z^{k_i} = 0$ (again, Proposition~\ref{DualConvergence}), we have
$\lim_{i \to \infty} L_\rho(x^{k_i},z^{k_i},\omega^{k_{i+1}}) = \zeta^{CLD}$ also. 
Next, at each $i$, and integers $j$ such that $k_i < j \le k_{i+1}$, observe that
$$
L_\rho(x^{k_i},z^{k_i},\omega^{k_{i+1}}) \ge L_\rho(x^{j},z^{j},\omega^{j}) \ge L_\rho(x^{k_{i+1}},z^{k_{i+1}},\omega^{k_{i+1}}).
$$
In taking the limit of the above inequality as $i \to \infty$, it becomes evident that $\lim_{k \to \infty} L_\rho(x^{k},z^{k},\omega^{k}) = \zeta^{CLD}$ in the original sequence also.
By the optimality of $\overline{\omega}$ for problem~\eqref{P2C}, we know from Lemma~\ref{LemmaALOpt} that $\zeta_\rho^{AL}(\overline{\omega})=\zeta^{CLD}$, and so
each limit point $(\overline{x},\overline{z})$ must be optimal for problem~\eqref{EqQMPConv} with $\omega=\overline{\omega}$. Furthermore, by Lemma~\ref{LemmaALOpt}, $(\overline{x},\overline{z})$ must also be optimal for problem~\eqref{P1ConvB}. (These limit points exist furthermore, due to the compactness of $\conv(X)$ and the continuous and closed-form expression that the unique solution $z^k \in \argmin_z \braces{\norm{Qx^k-z}_2^2 : z \in Z}$ has given $x^k \in \conv(X)$ when $Z$ is a linear subspace.)
\end{proof}

\subsection{Parallelization and workload}
The opportunities for parallelization and distribution of the computational workload in SDM-GS-ALM, as stated in Algorithm~\ref{AlgDualAscentGSSSC}, are not immediately apparent.
This subsection explicitly indicates which update problems may be solved in parallel, and the nature of the required communication between the parallel computational nodes.

The bulk of computational work, parallelization, and parallel communication occurs within the SDM-GS method stated in Algorithm~\ref{AlgBCD-XZ}, where for the problems of interest, the following decomposable structures apply: $X=\prod_{i=1}^m X_i$, $D=\prod_{i=1}^m D_i$,  and $F(x,z)=\sum_{i=1}^m F(x_i,z)$. 
In the larger context of Algorithm~\ref{AlgDualAscentGSSSC}, the subproblem of Line~\ref{SDMGSXUpdate} in Algorithm~\ref{AlgBCD-XZ} can be solved in parallel given fixed $\widetilde{z} \in Z$ and $\omega \in Z^\perp$ along the block indices $i=1,\dots,m$ as
\begin{equation}\label{Alg4QMP}
\min_{x} \braces{ f_i(x) + (\omega_i)^\top Q_i x + \frac{\rho}{2} \norm{Q_i x - \widetilde{z}_i}_2^2 : x \in D_i},
\end{equation}
while the subproblem of Line~\ref{SCGDirFindingSP} is solved as
$$
\min_x \braces{ \nabla_x f_i(\widetilde{x}_i) + \left( \omega_i + \rho(Q_i \widetilde{x}_i - \widetilde{z}_i) \right)^\top Q_i x : x \in X_i}.
$$
\begin{remark}
In the setting where problem~\eqref{P1} is a large-scale mixed-integer linear optimization problem, the subproblems of Line~\ref{SDMGSXUpdate} are continuous convex quadratic optimization problems for each block $i=1,\dots,m$, which can be solved independently of one another and in parallel. In the same setting, the Line~\ref{SCGDirFindingSP} subproblems are mixed-integer optimization problems for each block $i=1,\dots,m$, which can also be solved independently of one another and in parallel.
Additionally, the reconstruction of $D$ occurring in Line~\ref{LineDb} can be done in parallel for each $D_i$ along the indices $i=1,\dots,m$.
\end{remark}

Parallel communication is needed for the computation of the $z$ update in Line~\ref{SDMGSZUpdate} in Algorithm~\ref{AlgBCD-XZ}. In the larger context of Algorithm~\ref{AlgDualAscentGSSSC}, this takes the form of solving
$$
\min_z \braces{\sum_{i=1}^m \norm{Q_i \widetilde{x}_i - z_i}_2^2 : z \in Z}.
$$
This is solved as an averaging that requires the reduce-sum type parallel communication.
The computation of values required to compute $\gamma^k$ in Line~\ref{ComputeCritVal} in Algorithm~\ref{AlgDualAscentGSSSC} also requires a reduce-sum type parallel communication. 
For implementation purposes, the computation of these values, including the computation of $\Gamma$ from the SDM-GS call, can be combined into one reduce-sum communication. In total, each iteration of Algorithm~\ref{AlgDualAscentGSSSC} requires two reduce-sum type communications, one for computing the $z$-update of Line~\ref{SDMGSZUpdate} Algorithm~\ref{AlgBCD-XZ}, and one combined reduce-sum communication to compute scalars associated with the Lagrangian bounds and the critical values for the termination conditions. 
The storage and updates of $x^k$ and $\omega^k$ and $D$ can also be done in parallel, while $z^k$ and $\gamma^k$ need to be computed and stored by every processor at each iteration $k$.

\section{Computational experiments and results}\label{Sect4}
In this section, we present and examine the results of experiments for two tests each with the following purpose.
\begin{description}
\item [{\bf Test 1:}] to demonstrate the effect of enforcing the serious step condition on the Lagrangian values;
\item [{\bf Test 2:}] to compare the parallel speedup between the use of two parallel implementations of SDM-GS-ALM (Algorithm~\ref{AlgDualAscentGSSSC}) and the two parallel approaches in~\cite{LubinEtAl2013}. Additionally, the final iteration Lagrangian bounds are compared between the different parallel implementations for each experiment.
\end{description}

Computational experiments were performed on instances from two classes of problems. The first class consists of the capacitated allocation problems (CAP)~\cite{Bodur2014EtAl}.
The second class consist of problems from the Stochastic Integer Programming Test Problem Library (SIPLIB), which are described in detail in~\cite{NtaimoPhD2004,SIPLIB} and accessible at~\cite{SIPLIB}. These are all large-scale mixed-integer linear optimization problems, so the preceding observations for when $f$ is linear apply.

Test 1 was conducted with a Matlab 2012b~\cite{MATLAB2012B} serial implementation of Algorithm~\ref{AlgDualAscentGSSSC} using CPLEX 12.6.1~\cite{CPLEX12-6} as the solver. The computing environment was on an Intel$^{\circledR}$ Core$^{\texttrademark}$ i7-4770 3.40 GHz processor with 8 GB RAM and on a 64-bit operating system. All experiments for Test 1  were run with maximum number of iterations $k_{max}=20$.
The parallel experiments of Test 2 were conducted with a
C++ implementation of Algorithm~\ref{AlgDualAscentGSSSC} using CPLEX 12.5~\cite{CPLEX12-5} as the solver and the message passing interface (MPI) for parallel communication. For reading SMPS files into scenario-specific subproblems and for their interface with CPLEX, we used modified versions of the COIN-OR~\cite{COIN-ORURL} Smi and Osi libraries, either to instantiate appropriate C++ class instances of the subproblems directly, or to write scenario-specific MPS files from the SMPS file. 
The computing environment for the Test 2 experiments is the Raijin cluster maintained by Australia's National Computing Infrastructure (NCI) and supported by the Australian government~\cite{NCIURL}.  
The Raijin cluster is a high performance computing (HPC) environment which has 3592 nodes (system units), 57472 cores of Intel Xeon E5-2670 processors with up to 8 GB PC1600 memory per core (128 GB per node). All experiments were conducted using one thread per CPLEX solve.

The results of the Test 1 set of experiments are depicted in the plots of Figure~\ref{FigSSC} (with additional Figures~\ref{FigCAP101},~\ref{FigSSLP5-25-50} and~\ref{FigSSLP10-50-100} in Appendix~\ref{AppAddFigs}).
The use of different penalty parameter $\rho$ values is differentiated by the use of different plot colors. 
The penalties are chosen so that the smallest penalties (in red) are near optimal in terms of the resulting computational performance, while the larger penalties are known beforehand to be too large for optimal performance. For testing purposes, this is the most interesting way to choose penalty values, as smaller (than optimal) penalty values yield very little difference in Lagrangian bound between the use of different SSC parameter values.
Solid line and dashed line plots depict the Lagrange bounds due to the use of a more stringent SSC parameter value $\gamma=0.5$ and
a more lenient value for the SSC parameter $\gamma = 0.125$, respectively. 
The dotted line plots depict the Lagrangian values resulting from the non-use of the SSC, so that it evaluates true no matter what. 
The following observations are suggested from the results of these Test 1 experiments:
\begin{enumerate}
\item First, the most significant differences between the varied use of SSC occur when the penalty coefficient values are large.
In this setting, it seems to be the case that the use of more stringent (i.e., larger) values of the SSC parameter $\gamma$ has the effect of mitigating the destabilizing effect of having a penalty parameter $\rho$ value that is too large. This is significant because the performance of iterative Lagrangian dual solution approaches based on (or related to) proximal bundle methods is sensitive to the tuning of the $\rho$ value, and the optimal tuning of such parameters is assumed to be unknown beforehand in practical applications. For this reason, any mechanism to mitigate the effect of having an unfavorable tuning of the penalty parameter is highly desirable.
\item As is the case for the proximal bundle method, information from the SSC test can be used to dynamically fine-tune the value of the penalty parameter $\rho$. For the convergence analysis culminating in Proposition~\ref{PropFWBCD} to remain valid, it is expected that if $\rho$ does vary with iteration $k$, that it should stabilize to some positive value.
\item While not enforcing the SSC can adversely affect the growth trend in the Lagrangian bound, the use of a SSC parameter $\gamma$ value that is too large can have a similar effect for the tail-end values. This is most clearly seen in the Figure~\ref{FigSSC} DCAP-233-500 $\rho=50$ and $\rho=100$ plots. 
In these plots, the growth in Lagrangian bound value is noticeably stunted in the tail-end iterations for the larger $\gamma=0.5$ value as compared with the smaller $\gamma=0.125$.
\end{enumerate}

\begin{figure}[hbtp]        
\begin{tabular}{c}
\includegraphics[trim = 30mm 15mm 25mm 10mm, clip,width=1\textwidth]{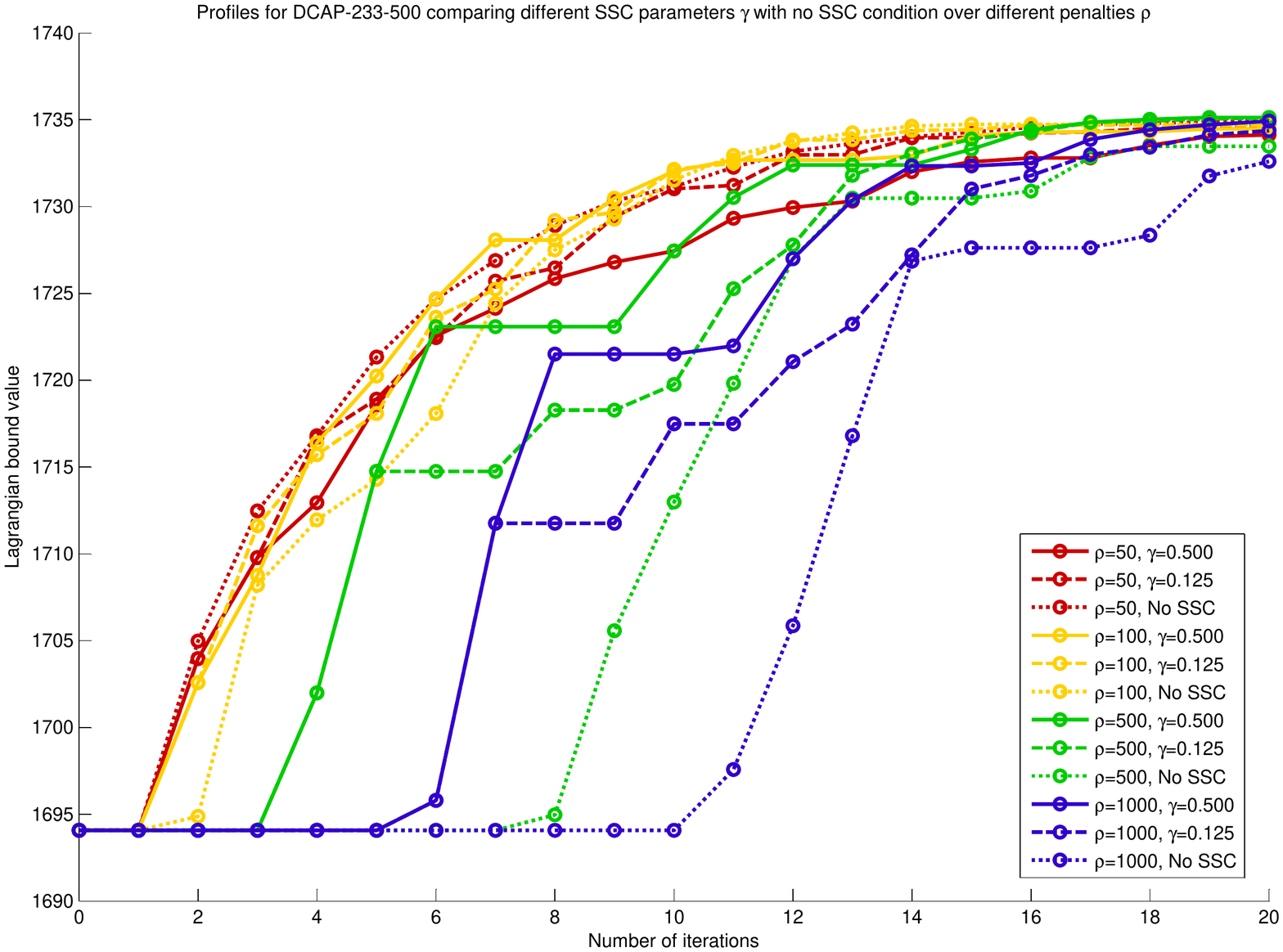}   
\end{tabular}
\caption{Applying SDM-GS-ALM using different parameterizations for the SSC condition (or none).\label{FigSSC}}
\end{figure}


For the Test 2 experiments, we primarily compare the parallel speedup achieved with Algorithm~\ref{AlgDualAscentGSSSC} against that achieved with
the enhancements to the proximal bundle method presented in~\cite{LubinEtAl2013}. Additionally, we compare the Lagrangian bound at the final iteration.

The enhancements in~\cite{LubinEtAl2013} use structure-exploiting primal-dual interior point solvers to improve the parallel efficiency of solving the proximal bundle method master problem. (The solution of this master problem is analogous to the approximated solution to problem~\eqref{EqQMPConv} obtained by using the SDM-GS method in Algorithm~\ref{AlgBCD-XZ}.)
The first solver is referred to by its acronym OOQP~\cite{GertzWright2003}, while the second is PIPS-IPM~\cite{LubinEtAl2011}.

In the experiments of Test 2, the underlying computing architecture and third-party software are inevitably different between our tests and those in~\cite{LubinEtAl2013}.
Additionally, the termination criterion is necessarily different from that given in Step 2 of Figure 2 in~\cite{LubinEtAl2013} due to the differences in algorithms. In our tests, the termination criterion comes from Lines~\ref{Alg4BeginTermCond}--\ref{Alg4EndTermCond} of Algorithm~\ref{AlgDualAscentGSSSC} with $\epsilon = 10^{-6}$.
We can nevertheless create a meaningful control in the tuning of the most important parameters affecting the performance of the algorithm.
\begin{enumerate}
\item As done in ~\cite{LubinEtAl2013}, we set the SSC parameter $\gamma=0.1$, and we initialize the dual solution $\omega^0 = 0$.
\item 
In analogy to the possible trimming of cutting planes noted in~\cite{LubinEtAl2013}, practical implementations of Algorithm~\ref{AlgDualAscentGSSSC} may judiciously trim the set $D$ to improve performance. As all cuts are kept in the experiments of~\cite{LubinEtAl2013}, so we also avoid trimming the expansion of $D$ in our experiments, and so we just use the simple update rule $D \gets \conv(D \cup \braces{\widetilde{x},\widehat{x}})$ within Algorithm~\ref{AlgBCD-XZ}.
\item We use an update rule analogous to the one in~\cite{Kiwiel1995} as is done in~\cite{LubinEtAl2013}. which takes the suggested form
given in Line~\ref{Alg4RhoUpdate} of Algorithm~\ref{AlgDualAscentGSSSC}. Initially, $\rho = 1$.
\end{enumerate}

In Tables~\ref{TabParallelSpeedupSSLP}--\ref{TabParallelSpeedupDCAP}, the columns headed by OOQP and PIPS-IPM report the parallel speedup due to the use of $N=1,8,16,32$ processors, which are originally reported in Figure 2 of~\cite{LubinEtAl2013}. If, given the use of $N$ processors, $T_N$ denotes the total wall clock time (in seconds) divided by number of iterations, then we compute the parallel speedup as $T_1/T_N$. For the computational experiments with Algorithm~\ref{AlgDualAscentGSSSC}, we compute each table entry $T_1/T_N$ after taking, from five identically parameterized experiments, 1) the minimum $T_1$ value, and 2) the average $T_N$, $N > 1$, value. 
The column headed by SDM-GS1-ALM presents 
the parallel speedup values for the application of Algorithm~\ref{AlgDualAscentGSSSC} with $t_{max}=1$. The column headed by SDM-GS5-ALM is analogous, with $t_{max}=5$. The total wall clock time per iteration values used to compute the ratios $T_1/T_N$ are provided in Appendix~\ref{AppAddTabs}, accounting for taking the minimum ($N=1$) or average ($N>1$) over the five experiments for each set of parameterizations associated with Algorithm~\ref{AlgDualAscentGSSSC}. 
For the two sets of experiments based on the application of Algorithm~\ref{AlgDualAscentGSSSC}, a problem-specific maximum number of main loop iterations was set so as to make the tests as comparable with the tests in~\cite{LubinEtAl2013} as possible. These data are also reported in Appendix~\ref{AppAddTabs}.
Also in Tables~\ref{TabParallelSpeedupSSLP}--\ref{TabParallelSpeedupDCAP}, the best Lagrangian bounds obtained for each combination of test problem and algorithm are reported.

\begin{table}[hbtp]
\begin{footnotesize}
\begin{tabular}{|c|rr|rr|}
\hline
& \multicolumn{4}{c|}{Speedup for SSLP 5-25-100}\\
No. Proc.  & \multicolumn{1}{c}{OOQP} & \multicolumn{1}{c}{PIPS-IPM} & \multicolumn{1}{c}{SDM-GS1-ALM} & \multicolumn{1}{c|}{SDM-GS5-ALM} \\
\hline
1		&1.00				&1.00				&1.00				&1.00\\
8		&5.54				&5.23				&4.38				&4.78\\
16		&8.89				&8.55				&6.61				&7.07\\
32		&11.69			&11.94			&8.19				&8.89\\
\hline
 Lagr. Value  & \multicolumn{1}{r}{-127.37} & \multicolumn{1}{r|}{-127.37} & \multicolumn{1}{r}{-127.71}  & \multicolumn{1}{r|}{-127.58} \\ 
 \hline
 \hline
 &\multicolumn{4}{|c|}{Speedup for SSLP 10-50-500}\\
 No. Proc.  & \multicolumn{1}{c}{OOQP} & \multicolumn{1}{c}{PIPS-IPM} & \multicolumn{1}{c}{SDM-GS1-ALM} & \multicolumn{1}{c|}{SDM-GS5-ALM} \\
\hline
1			&1.00			&1.00				&1.00				&1.00\\
8			&2.64			&2.80				&6.87				&6.95\\
16			&2.70			&2.92				&12.95			&12.84\\
32			&2.98			&3.40				&21.67			&20.98\\
\hline
 Lagr. Value  & \multicolumn{1}{r}{-349.14} & \multicolumn{1}{r|}{-349.14} & \multicolumn{1}{r}{-349.48}  & \multicolumn{1}{r|}{-349.14} \\ 
 \hline
 \hline
  & \multicolumn{4}{c|}{Speedup for SSLP 10-50-2000}  \\
No. Proc.  &  \multicolumn{2}{c}{SDM-GS1-ALM} & \multicolumn{2}{c|}{SDM-GS5-ALM} \\
 \hline
 1	&\multicolumn{2}{c}{1.00}		&\multicolumn{2}{c|}{1.00}\\
2	&\multicolumn{2}{c}{2.34}		&\multicolumn{2}{c|}{2.34}\\
4	&\multicolumn{2}{c}{4.81}		&\multicolumn{2}{c|}{4.83}\\
8	&\multicolumn{2}{c}{9.29}		&\multicolumn{2}{c|}{9.25}\\
16	&\multicolumn{2}{c}{18.69}		&\multicolumn{2}{c|}{18.48}\\
32	&\multicolumn{2}{c}{34.63}		&\multicolumn{2}{c|}{35.10}\\
64	&\multicolumn{2}{c}{60.59}		&\multicolumn{2}{c|}{60.93}\\
\hline
 Lagr. Value  & \multicolumn{2}{c}{-348.35}  & \multicolumn{2}{c|}{-347.75} \\ 
 \hline
\end{tabular}
\caption{SSLP: Comparing speedup and final best Lagrangian bound } \label{TabParallelSpeedupSSLP}
\end{footnotesize}
\end{table}

\begin{table}[hbtp]
\begin{footnotesize}
\begin{tabular}{|c|rr|rr|}
\hline
& \multicolumn{4}{|c|}{Speedup for DCAP 233-500}\\
 \hline
No. Proc.  & \multicolumn{1}{c}{OOQP} & \multicolumn{1}{c}{PIPS-IPM} & \multicolumn{1}{c}{SDM-GS1-ALM} & \multicolumn{1}{c|}{SDM-GS5-ALM} \\
\hline
1		&1.00			&1.00 		&1.00 & 	1.00\\
8		&2.44			&5.32 		&6.88 	& 8.11\\
16		&2.81			&8.15 		&13.28 	& 15.65\\
32		&1.63			&10.25 		&23.42 	& 27.40\\
\hline
 Lagr. Value  & \multicolumn{1}{r}{1736.68} & \multicolumn{1}{r|}{1736.68} & \multicolumn{1}{r}{1734.99}  & \multicolumn{1}{r|}{1736.02} \\ 
 \hline
 \hline
 &\multicolumn{4}{|c|}{Speedup for DCAP 243-500}\\
 \hline
No. Proc.  & \multicolumn{1}{c}{OOQP} & \multicolumn{1}{c|}{PIPS-IPM} & \multicolumn{1}{c}{SDM-GS1-ALM} & \multicolumn{1}{c|}{SDM-GS5-ALM} \\
\hline
1	& 	1.00	 &		1.00	&		1.00	&		1.00\\
8	&	2.85	&		5.71	&		6.51	&		7.61 \\
16	&	3.59	&		5.85	&		12.28	&		14.44 \\
32	&	1.98	&		6.44	&		21.99	&		25.25	\\				
\hline
 Lagr. Value  & \multicolumn{1}{r}{2165.48} & \multicolumn{1}{r|}{2165.50} & \multicolumn{1}{r}{2162.58}  & \multicolumn{1}{r|}{2164.48} \\ 
 \hline
 \hline
& \multicolumn{4}{|c|}{Speedup for DCAP 332-500}\\
 \hline
No. Proc.  & \multicolumn{1}{c}{OOQP} & \multicolumn{1}{c|}{PIPS-IPM} & \multicolumn{1}{c}{SDM-GS1-ALM} & \multicolumn{1}{c|}{SDM-GS5-ALM} \\
\hline
1			&	1.00			&	1.00		&	1.00		&	1.00  \\
8			&	2.03			&	5.56		&	6.83		&	8.50\\
16			&	2.33			&	5.00		&	12.84		&	16.20 \\
32			&	1.21			&	6.61		&	21.83		&	23.48\\
\hline
 Lagr. Value  & \multicolumn{1}{r}{1587.44} & \multicolumn{1}{r|}{1587.44} & \multicolumn{1}{r}{1584.77}  & \multicolumn{1}{r|}{1586.11} \\ 
 \hline
 \hline
 & \multicolumn{4}{|c|}{Speedup for DCAP 342-500}\\
 \hline
No. Proc.  & \multicolumn{1}{c}{OOQP} & \multicolumn{1}{c|}{PIPS-IPM} & \multicolumn{1}{c}{SDM-GS1-ALM} & \multicolumn{1}{c|}{SDM-GS5-ALM} \\
\hline
1		& 1.00	& 1.00		&1.00			&1.00			\\
8		& 2.45	& 3.78		&7.16			&8.25			\\
16		& 2.71	& 4.36		&12.95		&15.49		\\
32		& 1.84	& 4.64		&22.41		&26.93		\\
\hline
 Lagr. Value  & \multicolumn{1}{r}{1902.84} & \multicolumn{1}{r|}{1903.21} & \multicolumn{1}{r}{1900.81}  & \multicolumn{1}{r|}{1901.90} \\ 
 \hline
\end{tabular}
\caption{DCAP: Comparing speedup and final best Lagrangian bound \label{TabParallelSpeedupDCAP}}
\end{footnotesize}
\end{table}

We draw the following conclusions from the results of the Test 2 experiments reported in Tables~\ref{TabParallelSpeedupSSLP}--\ref{TabParallelSpeedupDCAP}.
\begin{enumerate}
\item The improvement in parallel speedup (SDM-GS-ALM columns) over either OOQP or PIPS-IPM is evident for all problems except for the one with the fewest number of scenarios (SSLP 5-25-100). 
\item Slightly inferior final Lagrange bounds reported for SDM-GS1-ALM ($t_{max}=1$)  are evident. This deficit is improved by using SDM-GS with $t_{max}=5$, as done for the SDM-GS5-ALM experiments. But even these bounds are usually not as good as the bounds obtained with OOQP or PIPS-IPM; this is due to their more exact solving of the master problem instances. This suggests that as the iterations $k \ge 1$ increase, it is advantageous to solve the continuous master problem with SDM-GS iterations using larger $t_{max}$ values.
\item Interestingly, parallel speedup is enhanced for SDM-GS5-ALM over SDM-GS1-ALM; although the latter yields lower average total wall clock time per iteration, the proportion of efficiently parallelizable work seems to increase in the former.
\end{enumerate}

For Test 2, we also tested the performance of Algorithm~\ref{AlgDualAscentGSSSC} on the SSLP 10-50-2000 problem, which is of substantially larger scale than the other test problems considered in this paper. Using $N=1,2,4,8,16,32,64$ processors, we see very good speedup, which suggests the realized benefit of distributing the use of memory. We also see that for such large-scale problems, the additional cost in time of performing more inner loop Gauss-Seidel iterations (larger $t_{max}$) becomes marginal, since the cost of solving the mixed-integer linear subproblems takes a larger share of the computational time.

\section{Conclusion and future work}\label{SectConclusion}
Our contribution is motivated by the goal of improving the efficiency of parallelization applied to iterative approaches for solving the Lagrangian dual problem of 
large scale optimization problems. These problems have nonlinear convex differentiable objective $f$, decomposable nonconvex constraint set $X$, and nondecomposable affine constraint set $Qx=z$ to which Lagrangian relaxation is applied. Problems of such a form include the split variable extensive form of mixed-integer linear stochastic programs as a special case. 
Implicitly, our approach refers to the convex hull $\conv(X)$ of $X$, and the assumed lack of known description of $\conv(X)$ needs to be addressed. Proximal bundle methods (alternatively in the form of the proximal simplicial decomposition method or stabilized column generation) are well-known for addressing the latter issue. In the former issue, that of exploiting the large scale structure to apply parallel computation efficiently, we develop a modified augmented Lagrangian (AL) method with approximate subproblem solutions that incorporates ideas from the proximal bundle method.

The approximation of subproblem solutions is based on an iterative approach that integrates ideas from the simplicial decomposition method (SDM) (for constructing inner approximations of $\conv(X)$) and the nonlinear block Gauss-Seidel method. It is the latter Gauss-Seidel aspect that is primarily responsible for enhancing the parallel efficiency that is observed in the numerical experiments. While convergence analysis of the integrated SDM-GS approach may be derived from slight modifications to results in~\cite{Bonettini2011}, for the sake of completeness and explicitness, we provide in the appendix a proof of optimal convergence of SDM-GS as it is applied within our algorithm under a standard set of conditions. A distinction between so-called ``serious'' steps and ``null'' steps, in analogy to the proximal bundle method, is also recovered. Once these aspects are successfully integrated, then the contribution is complete, where the beneficial stabilization associated with proximal point methods and the ability to apply parallelization more efficiently are both realized. The resulting algorithm developed in this paper is referred to as SDM-GS-ALM, which has similar functionality to the alternating direction method of multipliers (ADMM).

We performed numerical tests of two sorts. 
In Test 1, we examined the impact of varying the serious step condition parameter. We found that parameterizations that effect more stringent serious step conditions seem to have the effect of mitigating the early iteration instability due to penalty parameters that are too large. At the same time, the more stringent serious step condition parameterizations seemed to result in slower convergence to dual optimality in the tail-end. As is the case for proximal bundle methods, information obtained in the serious step condition tests may be used to beneficially adjust the proximal term penalty coefficient in   early iterations. 

In Test 2, we examined the efficiency of parallelization, measured by the speedup ratio, due to the use of the SDM-GS-ALM, compared versus pre-existing implementations of the proximal bundle method that use structure exploiting primal dual interior point methods to improve parallel efficiency. We saw in these results a promising increase in parallel efficiency due to the use of SDM-GS-ALM, where the increase in parallel efficiency is attributed primarily to the successful incorporation of Gauss-Seidel iterations. The results of the last problem tested, SSLP 10-50-2000, additionally suggested a benefit due to the ability of SDM-GS-ALM to distribute not just the workload, but also the use of memory. The vector of auxiliary variables $z$ is the only substantial block of data that needs to be stored and modified by all processors. In the context of stochastic optimization problems, this represents a modest communication bottleneck in proportion to the number of first-stage variables for two-stage problems, while for multistage problems, the amount of such data that must be stored by every processor and modified by parallel communication can increase exponentially with the number of stages.

Potential future improvements include the following. While a default implementation of SDM-GS-ALM would have one Gauss-Seidel iteration per SDM-GS call, the Lagrangian bounds reported from the Test 2 experiments suggest that an improved implementation would have early iterations use one Gauss-Seidel iteration per SDM-GS call, but steadily increase the number of Gauss-Seidel iterations per SDM-GS call for the later iterations. This results in better Lagrangian bounds at termination. While these extra Gauss-Seidel iterations require extra parallel communication, the additional wall clock time required becomes increasingly marginal for larger problems where the cost of solving the SDM linearized subproblems associated with expanding the inner approximation increasingly outweighs the cost associated with computing the approximate solution of the continuous master problem and any required parallel communications. 

A potentially large improvement to the speed of convergence, in terms of wall clock time, would be to incorporate into the analysis the degree to which the SDM linearized subproblem can be solved suboptimally and yet retain the optimal convergence. We expect that solving these subproblems exactly, particularly in the early iterations, is highly wasteful, and providing a theoretical basis for controlling the tolerance of solution inaccuracy would be of great value. Another potential avenue for future work is to extend the experimental analysis to multistage mixed-integer stochastic optimization problems and/or nonlinear problems, as the form of the problem addressed by SDM-GS-ALM is general enough to model these types of problems.


\appendix
\section{Technical lemmas for establishing optimal convergence of SDM-GS}\label{AppSDMGS}
 
Given initial $(x^0,z^0) \in X \times Z \subset \mathbb{R}^n \times \mathbb{R}^q$, we consider the generation of the sequence $\braces{(x^k,z^k)}$ with iterations computed using Algorithm~\ref{AlgFWGS}, whose target problem is given by
\begin{equation}\label{EqSCGGSP}
\min_{x,z} \braces{F(x,z) : x \in X, z \in Z},
\end{equation}
where $(x,z) \mapsto F(x,z)$ is convex and continuously differentiable over $X \times Z$, and sets $X$ and $Z$ are closed and convex, with $X$ bounded and $z \mapsto F(x,z)$ is inf-compact for each $x \in X$.

We define the directional derivative with respect to $x$ as
$$F_x'(x,z;d) := \lim_{\alpha \downarrow 0} \frac{F(x+\alpha d,z) - F(x,z)}{\alpha}.$$
Of interest is the satisfaction of the following local stationarity condition at $x \in X$:
\begin{align}
&F_x'({x},{z};d) \ge 0 \quad \text{for all}\; d \in X - \braces{{x}} \label{XStat}
\end{align}
for any limit point $(x,z)=(\bar{x},\bar{z})$ of some sequence $\braces{(x^k,z^k)}$ of feasible solutions to problem~\eqref{EqSCGGSP}.
For the sake of nontriviality, we shall assume that the $x$-stationarity condition~\eqref{XStat} never holds at $(x,z)=(x^k,z^k)$ for any $k \ge 0$. Thus, for each $x^k$, $k \ge 0$, there always exists a $d^k \in X-\braces{x^k}$ for which $F_x'(x^k,z^k;d^k) < 0$.

\medskip
\noindent
{\bf Direction Assumptions (DAs):} For each iteration $k \ge 0$, given $x^k \in X$ and $z^k \in Z$, we have $d^k$ chosen so that
1) $x^k + d^k \in X$; and 2) $F_x'(x^k,z^k;d^k) < 0$.

\medskip
\noindent
{\bf Gradient Related Assumption (GRA):} Given a sequence $\braces{(x^k,z^k)}$ with $\lim_{k \to \infty} (x^k,z^k) = (\overline{x},\overline{z})$, and a bounded sequence $\braces{d^k}$ of directions, then the existence of a direction $\overline{d} \in X - \braces{\overline{x}}$ such that
$F_{x}'(\overline{x},\overline{z};\overline{d}) < 0$ implies that
\begin{equation}\label{EqSGR}
\limsup_{k \to \infty} F_{x}'(x^k,z^k;d^k) < 0.
\end{equation}
In this case, we say that $\braces{d^k}$ is \emph{gradient related} to $\braces{x^k}$. This {gradient related condition} is similar to the one defined in~\cite{Bertsekas1999}.
The sequence of directions ${d^k}$ is typically gradient related to $\braces{x^k}$ by construction. (See Lemma~\ref{LemmaDir}.)
\medskip

To state the last assumption, we require the notion of an Armijo rule step length $\alpha^k \in (0,1]$ given $(x^k,z^k,d^k)$ and parameters $\beta,\sigma \in (0,1)$.
\begin{algorithm}[H]
\caption{Computing an Armijo rule step length $\alpha^k$ at iteration $k$.\label{AlgFWGS}}
\begin{algorithmic}[1]
\Function{ArmijoStep}{$F$, $x^k$, $z^k$, $d^k$, $\beta$, $\sigma$}  
       \State $\alpha^k \gets 1$ \label{ArmijoBegin}
       \While{$F(x^k+\alpha^k d^k,z^k) - F(x^k,z^k) > \alpha^k \sigma F_x'(x^k,z^k;d^k)$}\label{ArmijoWhileBegin}
         \State $\alpha^k \gets \beta \alpha^k$
       \EndWhile\label{ArmijoEnd}\label{ArmijoWhileEnd}
       \State {\bf return} $\alpha^k$
\EndFunction
\end{algorithmic}
\end{algorithm} 
\begin{remark}\label{ArmijoRuleRemark}
Under mild assumptions on $F$ such as continuity that guarantee the existence of finite 
$F_x'(x,z;d)$ 
for all $(x,z,d) \in \braces{(x,z,d) : x \in X, d \in X-\braces{x}, z \in Z}$, we may assume that the while loop of Lines~\ref{ArmijoWhileBegin}--\ref{ArmijoWhileEnd} terminates after a finite number of iterations. Thus, we have $\alpha^k \in (0,1]$ for each $k \ge 1$. 
\end{remark}
The last significant assumption is stated as follows.

\medskip
\noindent
{\bf Sufficient Decrease Assumption (SDA):} For sequences $\braces{(x^k,z^k,d^k)}$ and step lengths $\braces{\alpha^k}$ computed according to Algorithm~\ref{AlgFWGS}, we assume for each $k \ge 0$, that $(x^{k+1},z^{k+1})$ satisfies
$$
F(x^{k+1},z^{k+1}) \le F(x^k + \alpha^k d^k,z^k).
$$

\begin{lemma}\label{SCGXStatLemma}
For problem~\eqref{EqSCGGSP}, let $F : \mathbb{R}^{n_x} \times \mathbb{R}^{n_z} \mapsto \mathbb{R}$ be convex and continuously differentiable, $X \subset \mathbb{R}^{n_x}$ convex and compact, and $Z \subseteq \mathbb{R}^{n_z}$ closed and convex. Furthermore, assume for each $x \in X$ that $z \mapsto F(x,z)$ is inf-compact. If a sequence $\braces{(x^k,z^k,d^k)}$ satisfies the DA, the GRA, and the SDA for some fixed $\beta,\sigma \in (0,1)$, then the sequence ${(x^k,z^k)}$ has limit points $(\overline{x},\overline{z})$, each of which satisfies the stationarity condition~\eqref{XStat}.
\end{lemma}
\begin{proof}
The existence of limit points $(\overline{x},\overline{z})$ follows from the compactness of $X$, the inf-compactness of $z \mapsto F(x,z)$ for each $x \in X$, and the SDA.

In generating $\braces{\alpha^k}$ according to the Armijo rule as implemented in Lines~\ref{ArmijoBegin}--\ref{ArmijoEnd} of Algorithm~\ref{AlgFWGS}, we have
\begin{equation}\label{EqSCGGSArmijoRule}
\frac{F(x^k + \alpha^k d^k,z^k) - F(x^k,z^k)}{\alpha^k}  \le \sigma F_x'(x^k,z^k ; d^k).
\end{equation}
By the DA,  $F_x'(x^k,z^k ; d^k) < 0$ and since $\alpha^k > 0$ for each $k \ge 1$ by Remark~\ref{ArmijoRuleRemark}, we infer from~\eqref{EqSCGGSArmijoRule} that
$
F(x^k + \alpha^k d^k,z^k) < F(x^k,z^k).
$
By construction, we have
$
F(x^{k+1},z^{k+1}) \le  F(x^k + \alpha^k d^k,z^k) < F(x^k,z^k).
$
By the monotonicity $F(x^{k+1},z^{k+1}) < F(x^k,z^k)$ and $F$ being bounded from below on $X \times Z$, we have $\lim_{k \to \infty} F(x^k,z^k) = \bar{F} > -\infty$.
Therefore, $$\lim_{k \to \infty} F(x^{k+1},z^{k+1}) - F(x^k,z^k) = 0,$$ which implies 
\begin{equation}\label{EqSCGGSVanishing}
\lim_{k \to \infty} F(x^k + \alpha^k d^k,z^k) - F(x^k,z^k) = 0.
\end{equation}


We assume for sake of contradiction that $\lim_{k \to \infty} (x^k,z^k) = (\overline{x},\overline{y})$ does not satisfy the stationarity condition~\eqref{XStat}. 
By GRA, we have that $\braces{d^k}$ is gradient related to $\braces{x^k}$; that is, 
\begin{equation}\label{EqSCGGSSubgradRel}
\limsup_{k \to \infty} F_x'(x^k,z^k ; d^k) < 0.
\end{equation}
Thus, it follows from~\eqref{EqSCGGSArmijoRule}--\eqref{EqSCGGSSubgradRel} that $\lim_{k \to \infty} \alpha^k = 0$. 

Consequently, after a certain iteration $k \ge \bar{k}$, we can define $\braces{\bar{\alpha}^k}$, $\bar{\alpha}^k = \alpha^k / \beta$, where $\bar{\alpha}^k \le 1$ for $k \ge \bar{k}$, and so we have
\begin{equation}\label{EqReverseArmijoIneq}
\sigma F_x'(x^k,z^k ; d^k) < \frac{F(x^k + \bar{\alpha}^k d^k,z^k) - F(x^k,z^k)}{\bar{\alpha}^k}.
\end{equation}
Since $F$ is continuously differentiable, the mean value theorem may be applied to the right-hand side of~\eqref{EqReverseArmijoIneq} to get
\begin{equation}\label{EqReverseArmijoIneq2}
\sigma F_x'(x^k,z^k ; d^k) <  F_x'(x^k +\widetilde{\alpha}^k d^k, z^k ; d^k),
\end{equation}
for some $\widetilde{\alpha}^k \in [0,\overline{\alpha}^k]$.

Again, using the assumption $\limsup_{k \to \infty} F_x'(x^k,z^k ; d^k) < 0$, and also the compactness of $X - X$, we take a limit point $\overline{d}$ of $\braces{d^k}$, with its associated subsequence index set denoted by $\mathcal{K}$, such that $F_x'(\overline{x},\overline{z},\overline{d}) < 0$. 
Taking the limits over the subsequence indexed by $\mathcal{K}$, we have
$\lim_{k \to \infty, k \in \mathcal{K}} F_x'(x^k,z^k ; d^k) = F_x'(\overline{x},\overline{z} ; \overline{d})$ and $\lim_{k \to \infty, k \in \mathcal{K}} F_x'(x^k +\widetilde{\alpha}^k d^k, z^k ; d^k) = F_x'(\overline{x},\overline{z} ; \overline{d})$. These two limits holds since 1) $(x,z) \mapsto F_x'(x,z;d)$ for each $d \in X - X$ is continuous and 2) $d \mapsto F_x'(x,z;d)$ is locally Lipschitz continuous for each $(x,z) \in X \times Z$ (e.g., Proposition 2.1.1 of~\cite{Clarke1990}); these two facts together imply that $(x,z;d) \mapsto F_x'(x,z;d)$ is continuous. 
Then,  inequality~\eqref{EqReverseArmijoIneq2} becomes in the limit as $k \to \infty$, $k \in \mathcal{K}$,
\begin{align*}
&\sigma F_x'(\overline{x},\overline{z} ; \overline{d}) \le  F_x'(\overline{x},\overline{z} ; \overline{d}), \\
\Longrightarrow\quad &  0 \le (1-\sigma) F_x'(\overline{x},\overline{z} ; \overline{d}). \label{EqReverseArmijoIneq3}
\end{align*}
Since $(1-\sigma) > 0$ and $F_x'(\overline{x},\overline{z} ; \overline{d})<0$, we have a contradiction. Thus, $\overline{x}$ must satisfy the stationary condition~\eqref{XStat}.
\end{proof}

\begin{remark}
Noting that $F_x'(x^k,z^k;d^k) = \nabla_x F(x^k,z^k) d^k$ under the assumption of continuous differentiability of $F$, one means of constructing $\braces{d^k}$ is as follows:
\begin{equation}\label{DirFindingSPGS}
d^k \gets \argmin_d \braces{\nabla_x F(x^k,z^k) d : d \in X-\braces{x^k}}.
\end{equation}
\end{remark}

\begin{lemma}\label{LemmaDir}
Given sequence $\braces{(x^k,z^k)}$ with $\lim_{k \to \infty} (x^k,z^k) = (\overline{x},\overline{z})$, let each $d^k$, $k \ge 1$, be generated as in~\eqref{DirFindingSPGS}.
Then $\braces{d^k}$ is gradient related to $\braces{x^k}$.
\end{lemma}
\begin{proof}
By the construction of $d^k$, $k \ge 1$, we have
$$
F_x'(x^k,z^k ; d^k) \le F_x'(x^k,z^k ; d) \quad \forall \; d \in X-\braces{x^k}.
$$
Taking the limit, we have
$$
\limsup_{k \to \infty} F_x'(x^k, z^k ; d^k) \le \limsup_{k \to \infty} F_x'(x^k, z^k ; d) \le F_x'(\overline{x}, \overline{z}; d) \quad \forall \; d \in X-\braces{\overline{x}},
$$
where the last inequality follows from the upper semicontinuity of the function $(x,z,d) \mapsto F_x'(x,z;d)$, which holds in our setting due, primarily, to Proposition 2.1.1 (b) of ~\cite{Clarke1990} given that $F$ is assumed to be convex and continuous on $\mathbb{R}^n$.
Taking $\overline{d} \in \argmin_d \braces{ F_x'(\overline{x},\overline{z} ; d) : d \in X - \braces{\overline{x}}}$, we have by the assumed nonstationarity that $F_x'(\overline{x},\overline{z};\overline{d}) < 0$.
Thus, 
$
\limsup_{k \to \infty} F_x'(x^k, z^k ; d^k) < 0,
$
and so GRA holds.
\end{proof}

\bibliographystyle{spmpsci}      
\bibliography{./bibliographyParallelSCG}   


\section{Supplementary Material: Additional Figures}\label{AppAddFigs}

%

\begin{figure}[H]
\includegraphics[trim = 30mm 10mm 30mm 10mm, clip, width=0.8\textwidth]{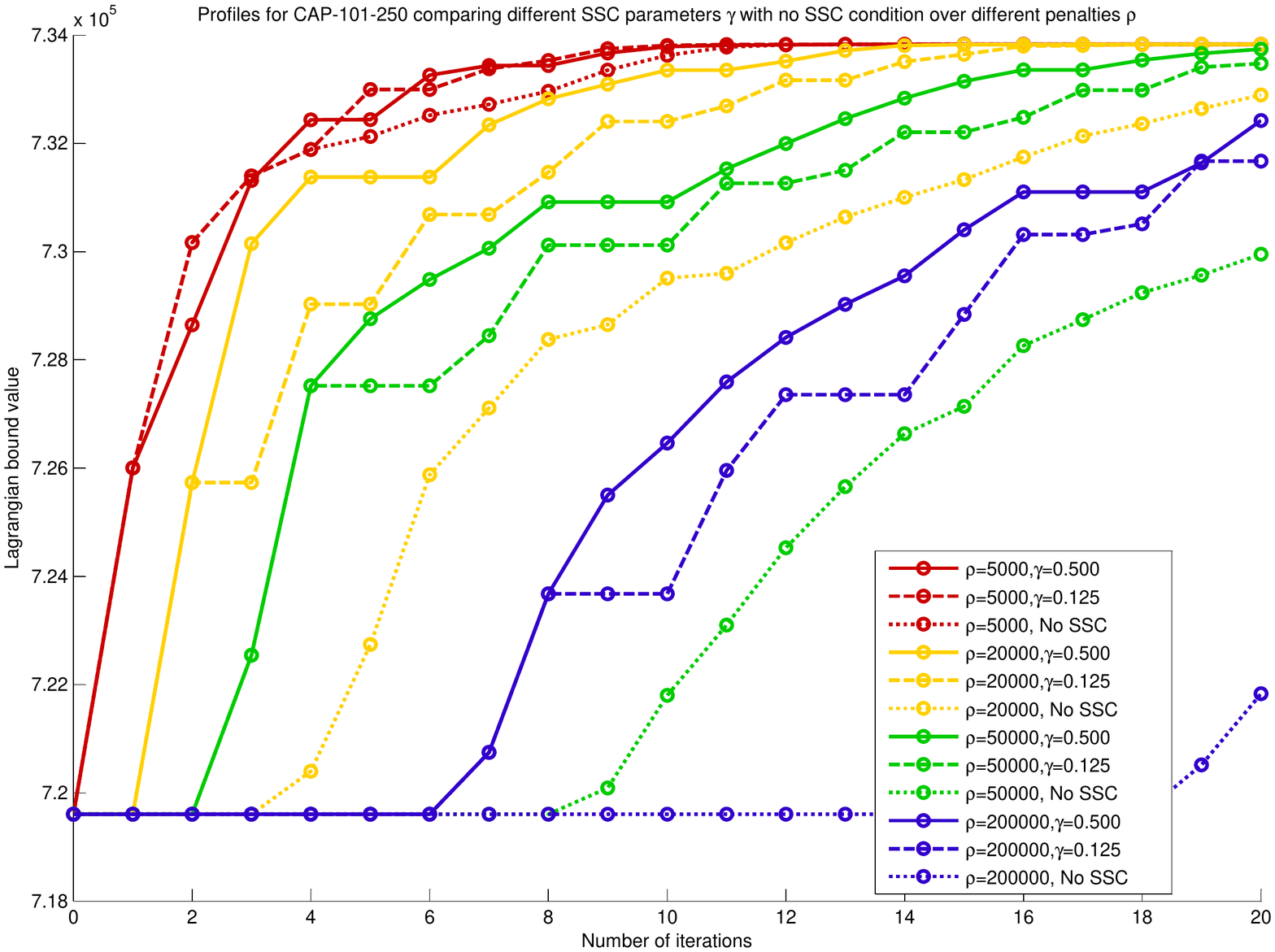} 
\caption{Applying SDM-GS-ALM to CAP-101-250 using different parameterizations for the SSC condition (or none).\label{FigCAP101}}
\end{figure}
\begin{figure}[H]
\includegraphics[trim = 30mm 10mm 25mm 10mm, clip,width=0.8\textwidth]{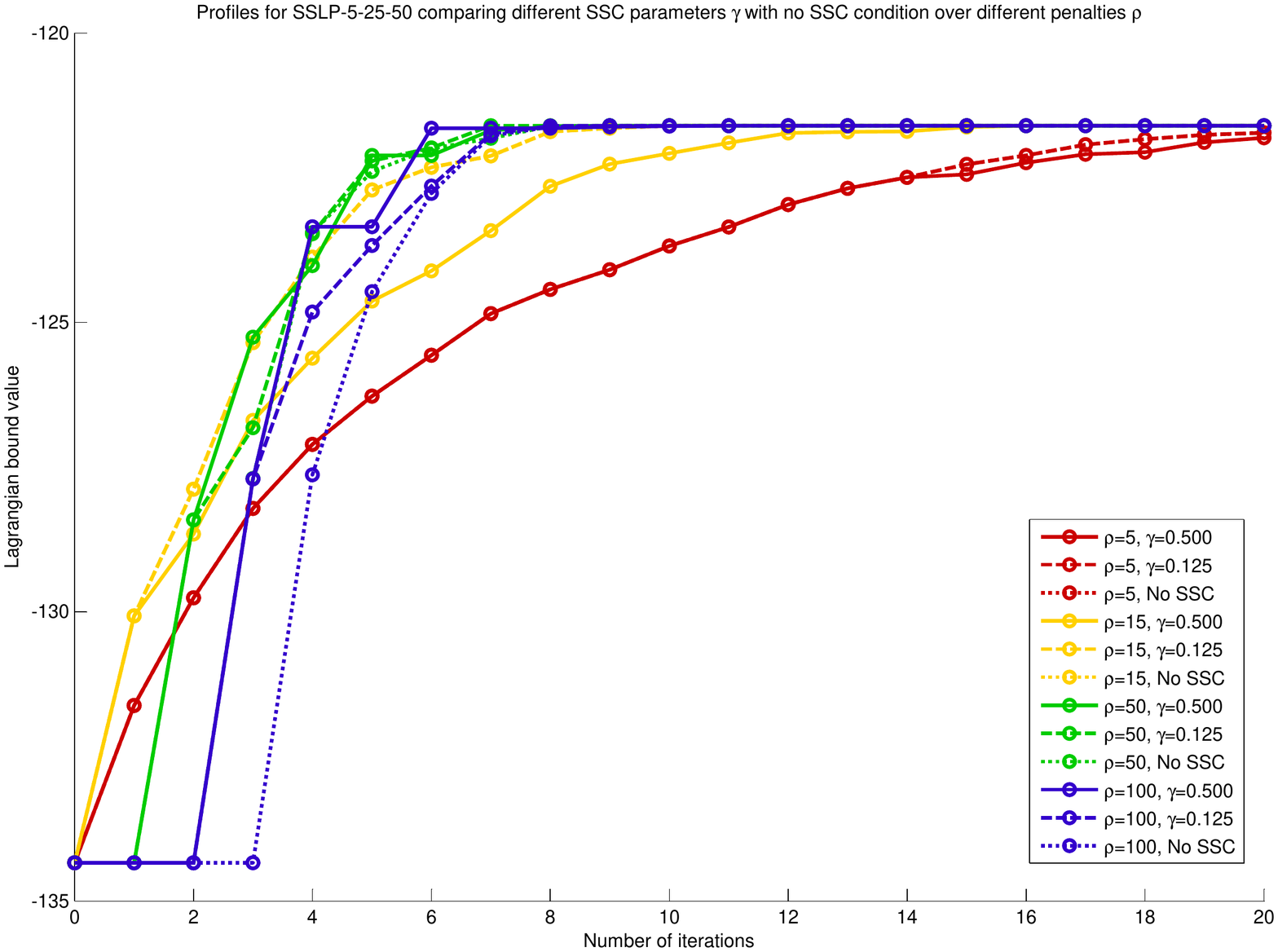}  
\caption{Applying SDM-GS-ALM to SSLP-5-25-50 using different parameterizations for the SSC condition (or none).\label{FigSSLP5-25-50}}
\end{figure}
\begin{figure}[H]
\includegraphics[trim = 30mm 20mm 25mm 10mm, clip,width=0.8\textwidth]{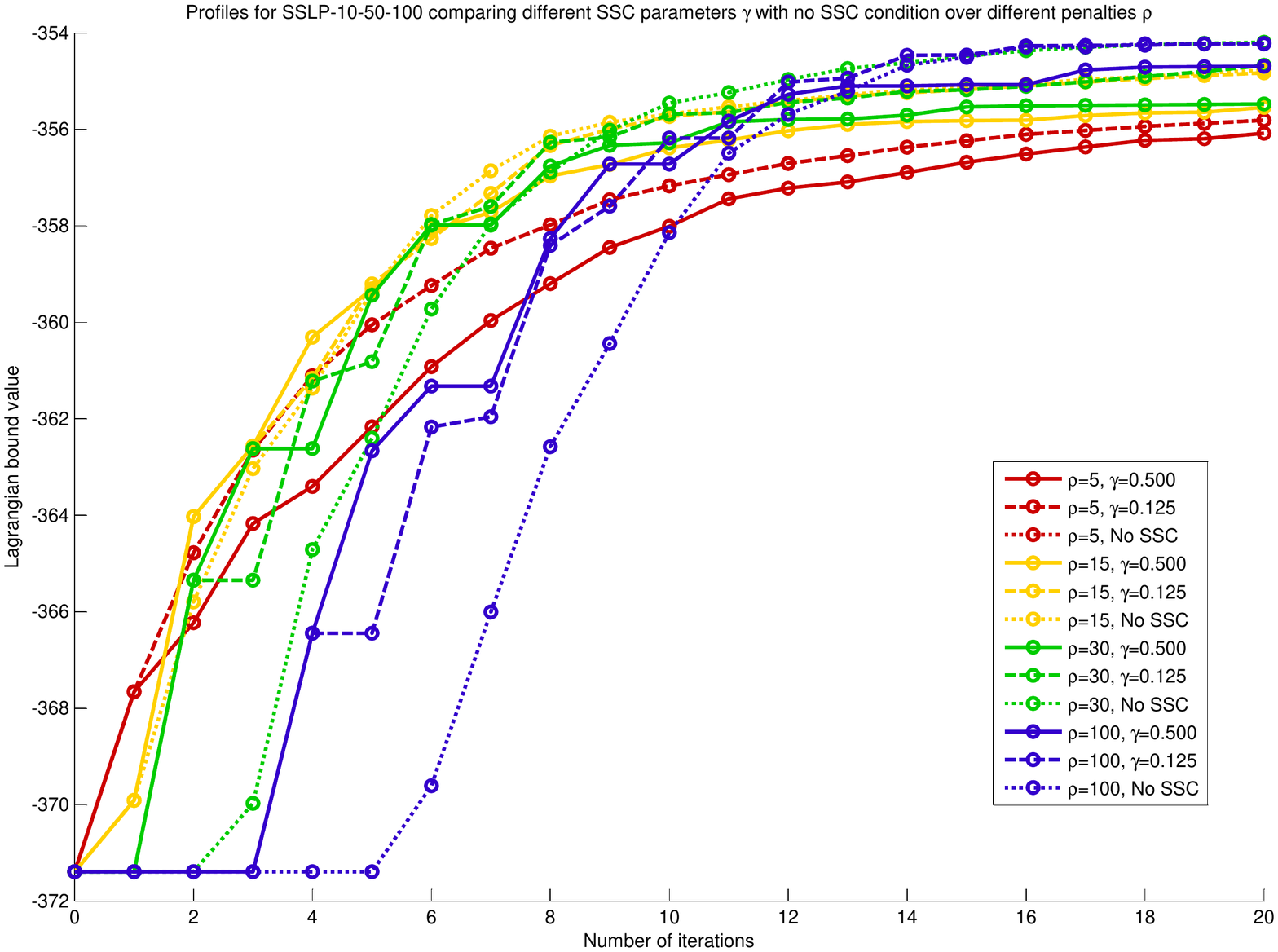} 
\caption{Applying SDM-GS-ALM to SSLP-10-50-100 using different parameterizations for the SSC condition (or none).\label{FigSSLP10-50-100}}
\end{figure}

\section{Supplementary Material: Additional Tables}\label{AppAddTabs}

For each entry $(A,B)$ of Tables~\ref{TabSupp1} and~\ref{TabSupp2}, $A$ provides the number of iterations at termination, and $B$ provides the average wall clock time (in seconds) per iteration.
\begin{table}[ht]
\begin{footnotesize}
\begin{tabular}{|c|rr|rr|}
\hline
& \multicolumn{4}{c|}{SSLP 5-25-100}\\
No. Proc.  & \multicolumn{1}{c}{OOQP} & \multicolumn{1}{c}{PIPS-IPM} & \multicolumn{1}{c}{SDM-GS1-ALM} & \multicolumn{1}{c|}{SDM-GS5-ALM} \\
\hline
1		& (8, 6.31)				& (8, 6.33)			& (8,3.22)			& (8,3.32)	\\
8		& (8, 1.14)				& (8, 1.21)			& (8,0.74)			& (8,0.69)	\\
16		& (8, 0.71)				& (8, 0.74)			& (8,0.49)			& (8,0.47)	\\
32		& (8, 0.54)				& (8, 0.53)			& (8,0.39)			& (8,0.37)	\\
\hline
\hline
 &\multicolumn{4}{|c|}{SSLP 10-50-500}\\
 No. Proc.  & \multicolumn{1}{c}{OOQP} & \multicolumn{1}{c}{PIPS-IPM} & \multicolumn{1}{c}{SDM-GS1-ALM} & \multicolumn{1}{c|}{SDM-GS5-ALM} \\
\hline
1		& (26, 3301)			& (22, 2939)			& (30, 168.80)		& (29, 171.85)	\\
8		& (31, 1252)			& (24, 1049)			& (30, 24.58)		& (29, 24.71)	\\
16		& (27, 1224)			& (28, 1005)			& (30, 13.04)		& (30, 13.39)	\\
32		& (31, 1106)			& (27, 865)			& (30, 7.79)			& (28, 8.19)		\\
\hline
\hline
  & \multicolumn{4}{c|}{SSLP 10-50-2000}  \\
No. Proc.  &  \multicolumn{2}{c}{SDM-GS1-ALM} & \multicolumn{2}{c|}{SDM-GS5-ALM} \\
 \hline
 1	&\multicolumn{2}{c}{(20, 840.69)}		&\multicolumn{2}{c|}{(20, 845.77)}\\
2	&\multicolumn{2}{c}{(20, 359.63)}		&\multicolumn{2}{c|}{(20, 361.40)}\\
4	&\multicolumn{2}{c}{(20, 174.83)}		&\multicolumn{2}{c|}{(20, 175.03)}\\
8	&\multicolumn{2}{c}{(20, 90.51)}		&\multicolumn{2}{c|}{(20, 91.46)}\\
16	&\multicolumn{2}{c}{(20, 44.98)}		&\multicolumn{2}{c|}{(20, 45.76)}\\
32	&\multicolumn{2}{c}{(20, 24.27)}		&\multicolumn{2}{c|}{(20, 24.09)}\\
64	&\multicolumn{2}{c}{(20, 13.87)}		&\multicolumn{2}{c|}{(20, 13.88)}\\
 \hline
\end{tabular}
\caption{SSLP: Reporting auxiliary data \label{TabSupp1}} 
\end{footnotesize}
\end{table}

\begin{table}[hb]
\begin{footnotesize}
\begin{tabular}{|c|rr|rr|}
\hline
& \multicolumn{4}{|c|}{DCAP 233-500}\\
 \hline
No. Proc.  & \multicolumn{1}{c}{OOQP} & \multicolumn{1}{c}{PIPS-IPM} & \multicolumn{1}{c}{SDM-GS1-ALM} & \multicolumn{1}{c|}{SDM-GS5-ALM} \\
\hline
1		& (68, 16.15)		& (66, 12.71) 	& (68, 3.67) 		& (68, 5.21)	\\
8		& (68, 6.62)			& (70, 2.39) 		& (68, 0.53) 		& (68, 0.64)	\\
16		& (68, 5.75)			& (73, 1.56) 		& (68, 0.28) 		& (68, 0.33)	\\
32		& (68, 9.91)			& (70, 1.24) 		& (68, 0.16) 		& (68, 0.19)	\\
\hline
\hline
 &\multicolumn{4}{|c|}{DCAP 243-500}\\
 \hline
No. Proc.  & \multicolumn{1}{c}{OOQP} & \multicolumn{1}{c|}{PIPS-IPM} & \multicolumn{1}{c}{SDM-GS1-ALM} & \multicolumn{1}{c|}{SDM-GS5-ALM} \\
\hline
1		& (57, 14.37)		& (57, 12.11) 	& (57, 3.72) 		& (57, 5.11)	\\
8		& (57, 5.04)			& (58, 2.12) 		& (57, 0.57) 		& (57, 0.67)	\\
16		& (57, 4.00)			& (59, 2.07) 		& (57, 0.30) 		& (57, 0.35)	\\
32		& (57, 7.26)			& (59, 1.88) 		& (57, 0.17) 		& (57, 0.20)	\\		
\hline
 \hline
& \multicolumn{4}{|c|}{DCAP 332-500}\\
 \hline
No. Proc.  & \multicolumn{1}{c}{OOQP} & \multicolumn{1}{c|}{PIPS-IPM} & \multicolumn{1}{c}{SDM-GS1-ALM} & \multicolumn{1}{c|}{SDM-GS5-ALM} \\
\hline
1		& (82, 13.51)		& (80, 9.45) 		& (82, 2.94) 		& (82, 4.85)	\\
8		& (82, 6.65)			& (79, 1.70) 		& (81, 0.43) 		& (82, 0.57)	\\
16		& (82, 5.81)			& (80, 1.89) 		& (81, 0.23) 		& (82, 0.30)	\\
32		& (82, 11.20)		& (77, 1.43) 		& (82, 0.13) 		& (82, 0.21)	\\	
\hline
 \hline
 & \multicolumn{4}{|c|}{DCAP 342-500}\\
 \hline
No. Proc.  & \multicolumn{1}{c}{OOQP} & \multicolumn{1}{c|}{PIPS-IPM} & \multicolumn{1}{c}{SDM-GS1-ALM} & \multicolumn{1}{c|}{SDM-GS5-ALM} \\
\hline
1		& (59, 14.78)			& (71, 12.07) 	& (59, 3.80) 		& (59, 5.57)	\\
8		& (59, 6.03)			& (67, 3.19) 		& (59, 0.53) 		& (59, 0.68)	\\
16		& (59, 5.46)			& (56, 2.77) 		& (59, 0.29) 		& (59, 0.36)	\\
32		& (59, 8.05)			& (62, 2.60) 		& (59, 0.17) 		& (59, 0.21)	\\	
\hline
\end{tabular}
\caption{DCAP: Reporting auxiliary data \label{TabSupp2}}
\end{footnotesize}
\end{table}

\end{document}